\documentclass[12pt]{article}
\usepackage{amsmath,amsthm,amsfonts,amssymb}

\newtheorem{theorem}{Theorem}[section]
\newtheorem{lemma}[theorem]{Lemma}
\newtheorem{cor}[theorem]{Corollary}
\newtheorem{prop}[theorem]{Proposition}
\newtheorem{assumption}[theorem]{Assumption}
\newtheorem{remark}[theorem]{Remark}

\numberwithin{equation}{section}

\newcommand {\IN}{\mathbb{N}}  
\newcommand {\IR}{\mathbb{R}}  

\newcommand {\cF}{\mathcal{F}}  
\newcommand {\cG}{\mathcal{G}}  
\newcommand {\cH}{\mathcal{H}}  

\newcommand {\rpl}{r_{q,\ell}}
\newcommand {\lsp}{\ell_q}
\usepackage{color}

\begin{document}

\title{Polarity of almost all points for systems of non-linear
stochastic heat equations in the critical dimension}
\author{Robert C.~Dalang$^1$
\\Ecole Polytechnique F\'ed\'erale de Lausanne
\and
Carl Mueller$^2$
\\University of Rochester
\and Yimin Xiao$^3$ \\
Michigan State University
}
\date{}
\maketitle

\footnotetext[1]{Supported in part by the Swiss National Foundation for Scientific Research.}
\footnotetext[2]{Supported by a Simons grant.}
\footnotetext[3]{Supported by NSF grants DMS-1607089 and DMS-1855185.



}

\begin{abstract}
We study vector-valued solutions $u(t,x)\in\mathbb{R}^d$ to systems of
nonlinear stochastic heat equations with multiplicative noise:
\begin{equation*}
\frac{\partial}{\partial t} u(t,x)=\frac{\partial^2}{\partial x^2} u(t,x)+\sigma(u(t,x))\dot{W}(t,x).
\end{equation*}
Here $t\geq0$, $x\in\mathbb{R}$ and $\dot{W}(t,x)$ is an $\IR^d$-valued space-time
white noise.  We say that a point $z\in\mathbb{R}^d$ is polar if
\begin{equation*}
P\{u(t,x)=z\text{ for some $t>0$ and $x\in\mathbb{R}$}\}=0.
\end{equation*}
We show that in the critical dimension $d=6$, almost all points in $\IR^d$ are polar.
\end{abstract}

\section{Introduction}
We say that a vector-valued stochastic process $(X_t,\, t\in I)$ hits a set $B$ if
\begin{equation*}
P\{X_t\in B\text{ for some $t\in I$}\}>0.
\end{equation*}
Hitting properties constitute one of the most intensively studied
topics in probability theory.  For many Markov processes,
probabilistic potential theory gives a powerful set of tools for
answering such questions \cite{BG,doob,khosh2}.  However, for processes taking
values in infinite dimensional spaces, potential theoretic
calculations are usually intractable and we must fall back on more
basic methods, such as covering arguments.

We also note that such hitting questions are always the most
difficult in the critical dimension, and we expect that hitting does
not occur in the critical dimension.  For example, if a family of vector-valued processes $(X_t^{(d)})$ can be defined so that for each $d \geq 1$, $X_t^{(d)}$ takes values in $\IR^d$, and if $B=\{z\}$ is a one-point set in $\IR^d$, we say
that $d_c$ is the critical dimension if hitting of $B$ occurs for
$d<d_c$ but not for $d> d_c$ (often, the superscript $d$ is omitted from the notation, as in \eqref{eq:she} and \eqref{eq:spde-main}).  For many natural families of such processes, we can often identify the
critical dimension $d_c$ even if we usually cannot prove that hitting of points fails to occur in that dimension.

In this paper, we deal with vector-valued solutions $u(t,x)$ to the
stochastic heat equation
\begin{align}
\label{eq:she}
\frac{\partial}{\partial t}u(t,x)
&=\frac{\partial^2}{\partial x^2}u(t,x)+\sigma(u(t,x))\dot{W}(t,x), \\
u(0,x)&=u_0(x),  \nonumber
\end{align}
where $x\in\mathbb{R}$,  $\dot{W}$ is a vector of $d$ independent
space-time white noises, and  $\sigma$ is matrix-valued.  We give more
precise conditions in the next section.

It would be possible to consider the solution $u(t,\cdot)$ as a stochastic process
parameterized by $t$ taking values in function space.  In view of the
difficulties mentioned above, we restrict ourselves to the question
of hitting points.  We say that $(u(t,\cdot),\, t \in\IR_+)$ hits the point $z \in \IR^d$ if
\begin{equation*}
P\left\{u(t,x)=z\text{ for some $t>0$ and $x\in\mathbb{R}$}\right\}>0.
\end{equation*}
So we are asking whether $(u(t,\cdot),\, t \in \IR_+)$ can hit the set $B$ of continuous functions $f(x)$ such that $f(x)=z$ for some value of $x\in\IR$.

Defined in this way, the question of hitting probabilities for
stochastic partial differential equations (SPDE) has been studied by
a number of authors, see \cite{mueller_tribe,DN,DKN1,DKN2,dal_sanz1}.  They obtain results for a
broad class of sets $B$, but only \cite{mueller_tribe,DN} deal with the critical dimension.
There are also some earlier papers about the question
of whether random fields can hit points or other sets.  For the vector-valued Brownian sheet, Orey and Pruitt \cite{op73} gave a necessary and sufficient condition for hitting points, and Khoshnevisan and Shi \cite{ks} developed a complete potential theory which answers the hitting question for any set.  Both groups of authors used special properties of the Brownian sheet.  For fractional Brownian fields, Talagrand \cite{T1,T2} answered the question of whether the process can hit points, including in the critical dimension. He also dealt with multiple points.

Building on Talagrand's methods, the article \cite{DMX} proved that for a broad class of Gaussian random fields, points are not hit in the critical dimension.  This paper also provided a general framework for this type of problem.  Some follow-up papers also deal with the question of multiple points for Gaussian random fields \cite{DLMX,DKLMX}.


In this paper, we deal with the nonlinear stochastic heat equation \eqref{eq:she} and show that in the critical dimension, which was known to be $d_c=6$ (see \cite{mueller_tribe,DKN2}), almost every point is polar, where ``almost every" is with respect to Lebesgue measure.  We will give a more precise statement
in the next section.  Because of the multiplicative noise term, the
equation is nonlinear and in most cases $u(t,x)$ will not be a
Gaussian process.  It is usually difficult to carry
over results about Gaussian processes to more general processes. However, it is well-known that on small scales,
$u(t,x)$ resembles a Gaussian process. By freezing in particular the coefficient $\sigma(u)$, it becomes possible to carry over many of the arguments
from \cite{DMX}.  However, we are still unable to prove that {\em all} points are polar in the critical dimension.
As part of our proof, we show that in dimensions
$d\geq6$, the 6-dimensional Hausdorff measure of the range of $u$ is 0.

For the linear heat equation, where $\sigma \equiv 1$ and the solution of \eqref{eq:she} is a Gaussian random field, the extra step that allows to go from ``almost all points are polar" to ``all points are polar" involves taking the conditional expectation of the random field given its value at a specific point (see \cite[Section 5]{DMX}). In the Gaussian case, conditional expectations can be computed explicitly, but in the nonlinear SPDE where $\sigma \not\equiv 1$, this is no longer true and a new argument seems to be needed.

We should also mention that because we can only show that almost
every point is polar in the critical dimension, we do not expect for the moment to be able to extend the results of this paper to the question of existence of multiple points and show that there are no multiple points in the critical
dimensions (except in the cases handled by \cite{DLMX,DKLMX} where $\sigma(u)$ is constant and so $u(t,x)$ is a Gaussian process). 


\section{Setup and main theorem}

Let $\sigma: \IR^d \to \IR^d \times \IR^d$ be a matrix function. We are dealing with solutions $u(t,x)$ to the system of $d$ equations
\begin{align}
\label{eq:spde-main}
\frac{\partial}{\partial t} u(t,x)&=\frac{\partial^2}{\partial x^2} u(t,x) +\sigma(u(t,x))\dot{W}(t,x), \qquad t >0,\ x \in \IR,
\end{align}
where $\dot W(t,x) = (\dot W_1(t,x),\dots,\dot W_d(t,x))$ is a $d$-dimensional space-time white noise (see \cite{khosh1}) defined on a probability space $(\Omega,\cF,P)$, with i.i.d.~components, subject to the initial condition
\begin{align}
u(0,x)&=u_0(x), \qquad   x \in \IR,
\end{align}
where $u_0: \IR \to \IR^d$ is Borel. We associate to the white noise its natural filtration $(\cF_t,\, t \in \IR_+)$, where $\cF_t$ is the $\sigma$-field generated by the white noise on $[0,t] \times \IR$ (and completed with $P$-null sets).

For an element $z \in \IR^d$, $\vert z \vert$ denotes the Euclidean norm of $z$. We use the same notation for a matrix $\sigma_0 \in \IR^{d \times d} \cong \IR^{d^2}$.

\begin{assumption}\label{assump2.1}
  (a) The function $\sigma: \IR^d \to \IR^{d \times d}$ is Lipschitz continuous with Lipschitz constant $L$: for all $v_1,v_2 \in \IR^d$,
\begin{equation*}
|\sigma(v_1)-\sigma(v_2)|\leq L|v_1-v_2|.
\end{equation*}
 (b) There is a finite constant $\sigma_1 \in \IR$ such that for all $v \in \IR^d$,
\begin{equation*}
|\sigma(v)|\leq\sigma_1.
\end{equation*}
(c) The initial function $u_0$ is bounded: there is $K_0 \in \IR_+$ such that, for all $x\in \IR$,
\begin{equation*}
|u_0(x)| \leq K_0.
\end{equation*}
\end{assumption}

Finally, we note that (\ref{eq:spde-main}) has a rigorous formulation in
terms of the mild form, see \cite{dal_sanz}: $(u(t,x),\ (t,x) \in \IR_+ \times \IR)$ is a jointly measurable and $(\cF_t)$-adapted process such that, for all $(t,x) \in \IR_+ \times \IR$,
\begin{equation*}
u(t,x)=\int_{-\infty}^{\infty}G(t,x-y)u_0(y)dy
 +\int_{0}^{t}\int_{-\infty}^{\infty}G(t-s,x-y)\sigma(u(s,y))W(dy,ds)
\end{equation*}
where
\begin{equation*}
G(t,x)=\frac{1}{\sqrt{4\pi t}}\exp\left(-\frac{x^2}{4t}\right)
\end{equation*}
is the heat kernel on $\IR$. Existence and uniqueness is proved in \cite[Chapter 3]{walsh} in the case $d=1$, and this proof extends directly to $d \geq 1$ (see \cite[Section 2]{DKN2}). The random field $(u(t,x))$ has a continuous version on $]0,\infty[\, \times \IR$ (see \cite{walsh}), and if the initial condition $u_0$ is continuous (which we do not assume here), then this version of $(u(t,x))$ is continuous on $\IR_+ \times \IR$ \cite[Theorem 3.1]{chen_dal}. We will work only with this continuous version.

The main result of this paper is the following. For the definition of Hausdorff measure, see \cite[Appendix C]{khosh2}.

\begin{theorem}\label{thm11.4}
Assume that $d\geq 6$. Almost surely, the range of $u=(u(t,x),\, (t,x) \in \,]0,\infty[ \times \IR)$ has $6$-dimensional Hausdorff-measure $0$. In particular, if $d\geq 6$, then (Lebesgue) almost all points in $\IR^d$ are polar for $u$.
\end{theorem}

This theorem is proved at the end of Section \ref{sec7}.

\begin{remark} For linear systems of stochastic heat equations ($\sigma$ constant), according to \cite{mueller_tribe} and \cite{DMX}, $d=6$ is the critical dimension for hitting points and points are polar when $d=6$. According to \cite[Corollary 1.5]{DKN2}, when the matrix function $\sigma$ is smooth and uniformly elliptic and $d>6$, then the Hausdorff dimension of the range of $u$ is precisely $6$. This implies of course that for $d>6$, almost all points in $\IR^d$ are polar. Therefore, the conclusions of Theorem \ref{thm11.4} are most interesting in the critical dimension $d=6$.
\end{remark}
\section{Local decomposition}\label{sec3}

In this section, our goal is to study the range of $(t,x) \mapsto u(t,x)$ when $(t,x)$ belongs to a small rectangle with center $(t_0,x_0) \in R_0 := [1,2]\times [0,1]$, where $t_0$ and $x_0$ are fixed. Throughout most of the paper, we will be working on subrectangles of $R_0$.

For $\rho \in \,]0,\tfrac{1}{2}]$, define
\begin{equation}\label{eqR0}
R_\rho = R_\rho(t_0,x_0) := \{(t,x)\in \IR_+ \times \IR: \vert t - t_0 \vert< \rho^4,\ \vert x - x_0 \vert  < \rho^2 \}.
\end{equation}
This rectangle has side-lengths that are compatible with the metric
\begin{equation*}
d((t,x);(s,y)) = \Delta(t-s,x-y) := \max(|t-s|^{1/4},|x-y|^{1/2}).
\end{equation*}

We often write $p$ for a couple $(t,x)\in\IR^2$.
Finally, we define the oscillation of an $\IR^d$-valued function $f$ on a rectangle $R\subset \IR^2$ as follows:
\begin{equation*}
\text{osc}_R(f)=\sup_{p_1,p_2\in R}|f(p_1)-f(p_2)|.
\end{equation*}


In a first stage, we would like to replace $(t,x) \mapsto u(t,x)$ by a modified process obtained by freezing coefficients at stopping times, so that certain  regularity and growth conditions are satisfied. We will also do this for an associated Gaussian process. For this, we define a first stopping time $\tau_{K,1}$ that will help with H\"older-continuity properties of the solution, then a stopping time $\tau_{K,2}$ that will deal with growth as $x \to \pm\infty$, and a third stopping time $\tau_{K,3}$ that will help with an associated Gaussian process.
\vskip 12pt

\noindent{\em First stopping time $\tau_{K,1}$}
\vskip 12pt

Fix $T_0 >3$ and a large constant $K >0$. From \cite[Theorem 3.1]{chen_dal}, we know that $u(t,x)$ is locally $(1-\delta)/4$-H\"older continuous in $t$ and $(1-\delta)/2$-H\"older continuous in $x$ on $]0,\infty[\,\times \IR$.  More precisely, for each $\delta\in\,]0,1[$, there is an almost surely finite positive random variable $Z$ such that for all $s,t \in [\tfrac{1}{2},T_0]$ and $x,y\in[-2,2]$,
\begin{equation}
\label{eq:u-modulus}
|u(t,x)-u(s,y)| \leq Z \Delta(t-s,x-y)^{1-\delta}.
\end{equation}

Now we define the stopping time $\tau_{K,1}$ to be the first time
$t\in[\tfrac{1}{2},T_0]$ such that there exist
$x,y\in  [-2,2]$ with
\begin{equation*}
|u(t,x)-u(s,y)|\geq K\Delta(t-s, x-y)^{1-\delta};
\end{equation*}
if there is no such time $t$, let $\tau_{K,1}=T_0$.

Now (\ref{eq:u-modulus}) shows that
\begin{equation}
\label{eq:stopping-time-1}
\lim_{K\to\infty}P\left\{\tau_{K,1}<T_0 \right\}=0.
\end{equation}
Also note that $u(t\wedge\tau_{K,1},x)$ satisfies
\begin{equation}\label{eq:3.4}
|u(t_1\wedge\tau_{K,1},x_1)-u(t_2\wedge\tau_{K,1},x_2)|\leq
K\Delta\big(t_1 -t_2,x_1-x_2 \big)^{1-\delta},
\end{equation}
for $(t_i,x_i)\in [\tfrac{1}{2},T_0] \times [-2,2]$.
\vskip 12pt

\noindent{\em Modified solution $\tilde{u}$}
\vskip 12pt

We will modify the random field $u$ using $\tau_{K,1}$. 
We define $\tilde{u}(t,x)=\tilde{u}_K(t,x)$ as the (continuous version on $]0,\infty[\,\times \IR$ of the) solution of
\begin{align*}
\frac{\partial}{\partial t} \tilde{u}(t,x)&= \frac{\partial^2}{\partial x^2} \tilde{u}(t,x)
 +\sigma(u(t\wedge\tau_{K,1},x))\dot{W}(t,x),  \qquad t >0,\ x \in \IR,  \\
\tilde{u}(0,x)&=u_0(x), \qquad x \in \IR.
\end{align*}
Note that on the right-hand side of the equation for $\tilde u$, $\sigma$ is evaluated at $u$, not at $\tilde{u}$.  In terms of the mild form,
\begin{align}\nonumber
\tilde{u}(t,x)&=\int_{-\infty}^{\infty}G(t,x-y)u_0(y)dy  \\
&\qquad +\int_{0}^{t}\int_{-\infty}^{\infty}G(t-s,x-y)
 \sigma(u(s\wedge\tau_{K,1},y))W(dy,ds).
 \label{e3.5}
\end{align}
Finally, note that on $\{\tau_{K,1}=T_0\}$,
we have that $u(t,x)=\tilde{u}(t,x)$ for all $(t,x)\in [0,T_0]\times \IR$.  Thus,
\begin{equation}
\label{eq:prob-equality-tilde}
\lim_{K\to\infty}P\left\{u(t,x)=\tilde{u}_K(t,x)
 \text{ for all }(t,x)\in [0,T_0]\times \IR\right\}=1.
\end{equation}
For the time being, we will work with $\tilde{u}$.
\vskip 12pt

\noindent{\em Second stopping time $\tau_{K,2}$}
\vskip 12pt

We also want to control the growth of our solution $\tilde u$ as $x \to \pm \infty$.  Let
$\tau_{K,2}$ be the first time $t\in[0,T_0]$ such that there exists $x\in\IR$ with
\begin{equation*}
|\tilde{u}(t,x)|\geq K(1+|x|).
\end{equation*}
If there is no such time $t$, let $\tau_{K,2}=T_0$.

Since we are assuming that $\sigma$ and our initial function $u_0(x)$ are bounded (Assumption \ref{assump2.1}(b) and (c)), it is a consequence of Lemma \ref{lem6.8} below (taking $\phi(r,z) = \sigma(u(r\wedge\tau_{K,1},z))$ in \eqref{eq:n-def-1} and $\phi_1 = \sigma_1$ in \eqref{eq:n-def-1a}) that
\begin{equation}
\label{eq:stopping-time-2}
\lim_{K\to\infty}P\left\{\tau_{K,2}<T_0 \right\}=0.
\end{equation}
\vskip 12pt

\noindent{\em Third stopping time $\tau_{K,3}$}
\vskip 12pt

We also work with the (continuous version on $]0,\infty[\,\times \IR$ of the) following linear system of stochastic heat equations with additive noise:
\begin{align}
\label{eq:spde-additive-noise}
\frac{\partial}{\partial t}  v(t,x)&= \frac{\partial^2}{\partial x^2} v(t,x)+\dot{W}(t,x), \qquad t >0,\ x \in \IR,  \\
v(0,x)&=u_0(x), \qquad  x \in \IR.  \nonumber
\end{align}
Now we define $\tau_{K,3}$ in the same way as
$\tau_{K,2}$, but with respect to $v$ rather than $\tilde u$:
$$
  \tau_{K,3} = T_0 \wedge \inf\{t \in [0,T_0]: \exists x \in \IR\mbox{ with } \vert v(t,x)\vert \geq K(1 + \vert x \vert) \} .
$$
As with the stopping time $\tau_{K,2}$, since we are assuming that our initial function $u_0(x)$ is bounded, it is a consequence of Lemma \ref{lem6.8} (taking $\phi(r,z) \equiv 1$ in \eqref{eq:n-def-1} and $\phi_1=1$ in \eqref{eq:n-def-1a}) below that
\begin{equation*}
\lim_{K\to\infty}P\left\{\tau_{K,3}<T_0\right\}=0.
\end{equation*}

\section{Local decomposition of the solution}
Fix
\begin{equation}\label{e4.1}
\alpha\in\,]\tfrac{1}{2},\tfrac{2}{3}[, \qquad  \beta\in\,]\alpha,\tfrac{2}{3}[. 
\end{equation}
Consider the rectangle $R_\rho(t_0,x_0)$ defined in \eqref{eqR0}. In order to study the behavior of $\tilde u$ in this rectangle, we are going to use a decomposition based on a time prior to $t_0 - \rho^4$, namely, we define
\begin{equation*}
t_0^-= t_0^-(\rho)=t_0 - \rho^4 -\rho^{4(1-\alpha)}
\end{equation*}
(notice that since $t_0 \geq 1$, $\rho\in \,]0,\tfrac{1}{2}]$ and $\alpha < \tfrac{2}{3}$, we have $t_0^- > \tfrac{1}{2}$). We also set
\begin{equation*}
   L_1 = L_1(\rho) = \rho^2 + \rho^{2(1-\beta)}.
\end{equation*}
The rectangle
\begin{equation*}
R^+=[t_0^-(\rho),t_0+\rho^4]\times[x_0-L_1,x_0 + L_1].
\end{equation*}
is an enlargement of $R_\rho(t_0,x_0)$. Note that for small $\rho >0$,
\begin{equation*}
\rho^{4(1-\alpha)}\ll \left[\rho^{2(1-\beta)}\right]^2;
\end{equation*}
the idea is that in the $x$-direction, $R^+$ is larger than parabolic scaling would indicate. Indeed, we would have exact parabolic scaling if $\beta$ were equal to $\alpha$.

\subsection{Isolating the dominant term}
We use the Markov property \cite[Chapter 9]{DZ} to start $\tilde{u}$ afresh at time $t_0^-$, so that
for $(t,x)\in R_\rho(t_0,x_0)$, we have
\begin{align*}
\tilde{u}(t,x)&= \tilde u_{t_0,x_0,\rho}(t,x) + N_{t_0,x_0,\rho}(t,x),
\end{align*}
where, for $t \geq t_0^-$ and $x\in \IR$,
\begin{align}\label{eq_utilde}
   \tilde u_{t_0,x_0,\rho}(t,x) &= \int_{-\infty}^{\infty}G(t-t_0^-,x-y)\tilde{u}(t_0^-,y)dy  \\ \nonumber
  N_{t_0,x_0,\rho}(t,x)  &=\int_{t_0^-}^{t}\int_{-\infty}^{\infty}
 G(t-s,x-y)\sigma(u(s\wedge\tau_{K,1},y))W(dy,ds).
\end{align}
Note that $\tilde u_{t_0,x_0,\rho}(t,x)$ is the solution of the heat equation started at time $t_0^-$ with initial function $\tilde u(t_0^-, \cdot)$.

We further decompose $N_{t_0,x_0,\rho}(t,x)$ as follows. Let
\begin{align}\nonumber
N^{(0)}(t,x)&=\int_{0}^{t}\int_{-\infty}^{\infty}G(t-s,x-y)W(dy,ds),  \\ \nonumber
N^{(1)}_{t_0,x_0,\rho}(t,x)&=\int_{t_0^-}^{t}\int_{x_0-L_1}^{x_0+L_1}G(t-s,x-y)
 \big[\sigma(u(s\wedge\tau_{K,1},y))-\sigma(u(t_0^-\wedge\tau_{K,1},x_0))\big]W(dy,ds),  \\ \nonumber
N^{(2)}_{t_0,x_0,\rho}(t,x)&=\int_{t_0^-}^{t}\int_{[x_0-L_1,x_0+L_1]^c}G(t-s,x-y)
 \big[\sigma(u(s\wedge\tau_{K,1},y))-\sigma(u(t_0^-\wedge\tau_{K,1},x_0))\big]W(dy,ds),   \\ \nonumber
v^{(1)}_{t_0,x_0,\rho}(t,x)&=\int_{0}^{t_0^-}\int_{-\infty}^{\infty}G(t-s,x-y)W(dy,ds)   \\
&=\int_{-\infty}^{\infty}G(t-t_0^-,x-y)N^{(0)}(t_0^-,y)dy.
\label{e4.3}
\end{align}
In the last line above, we have used semigroup property of $G$ and the stochastic Fubini theorem, see
\cite[Theorem 2.6]{walsh}; 
notice that the dependence of these processes on $K$ is omitted from the notation. Note also that $v_{t_0,x_0,\rho}^{(1)}$ is very similar to $\tilde u_{t_0,x_0,\rho}$.  We see that for $t \geq t_0^-$ and $x\in \IR$,
\begin{align*}
N_{t_0,x_0,\rho}(t,x)&=\sigma(u(t_0^-\wedge\tau_{K,1},0))N^{(0)}(t,x)+N_{t_0,x_0,\rho}^{(1)}(t,x)
 +N_{t_0,x_0,\rho}^{(2)}(t,x)    \\
&\qquad -\sigma(u(t_0^-\wedge\tau_{K,1},0))v_{t_0,x_0,\rho}^{(1)}(t,x).
\end{align*}
With the notation in \eqref{eq_utilde} and \eqref{e4.3} above, we have
\begin{align*}
\tilde{u}(t,x) &= \sigma(u(t_0^-\wedge\tau_{K,1},0))N^{(0)}(t,x) \\
   &\qquad + \tilde u_{t_0,x_0,\rho}(t,x) +N_{t_0,x_0,\rho}^{(1)}(t,x)
 +N_{t_0,x_0,\rho}^{(2)}(t,x) \\
  & \qquad  -\sigma(u(t_0^-\wedge\tau_{K,1},0))v_{t_0,x_0,\rho}^{(1)}(t,x).
\end{align*}

We also impose some growth conditions on $x\mapsto \tilde u_{t_0,x_0,\rho}(t,x)$ and $x \mapsto v_{t_0,x_0,\rho}^{(1)}(t,x)$ by defining, for $t \geq t_0^-$ and $x\in\IR$,
\begin{align}\nonumber
\hat {u}_{t_0,x_0,\rho}(t,x)&=
\begin{cases}
\tilde u_{t_0,x_0,\rho}(t,x) &\text{if }t_0^-\leq\tau_{K,2},  \\
0 &\text{if } t_0^->\tau_{K,2},
\end{cases}  \\
\hat{v}_{t_0,x_0,\rho}^{(1)}(t,x)&=
\begin{cases}
v^{(1)}_{t_0,x_0,\rho}(t,x) &\text{if }t_0^-\leq\tau_{K,3},  \\
0 &\text{if }t_0^->\tau_{K,3}.
\end{cases}
\label{e4.4}
\end{align}
Finally, for $t \geq t_0^-$ and $x\in \IR$, we define a new process $(w_{t_0,x_0,\rho}(t,x))$, which is related to the solution $u(t,x)$ but with frozen coefficients and controlled growth, by
\begin{equation}\label{eq:4.2}
   w_{t_0,x_0,\rho}(t,x) = \sigma(u(t_0^-\wedge\tau_{K,1},0))N^{(0)}(t,x) + E_{t_0,x_0,\rho}(t,x),
\end{equation}
where
\begin{align}\label{eq:4.2a}
  E_{t_0,x_0,\rho}(t,x) &= N^{(1)}_{t_0,x_0,\rho}(t,x) +N^{(2)}_{t_0,x_0,\rho}(t,x)
 -\sigma(u(t_0^-\wedge\tau_{K,1},0))\hat v^{(1)}_{t_0,x_0,\rho}(t,x) + \hat u_{t_0,x_0,\rho}(t,x).
\end{align}
Observe that if $\tau_{K,2}\wedge\tau_{K,3}=T_0$, then $\hat u_{t_0,x_0,\rho}\equiv\tilde{u}_{t_0,x_0,\rho}$ and
$\hat{v}_{t_0,x_0,\rho}^{(1)} \equiv v^{(1)}_{t_0,x_0,\rho}(t,x)$.  Thus, if $\tau_{K,2}\wedge\tau_{K,3}=T_0$, then  for $t \geq t_0^-$ and $x \in\IR$,
\begin{align}\label{eq:4.3}
\tilde{u}(t,x)&= w_{t_0,x_0,\rho}(t,x).
\end{align}

We wish to show that the oscillation of $\tilde u$ on the rectangle $R_\rho(t_0,x_0)$ is comparable to the oscillation of $N^{(0)}$ on $R_\rho(t_0,x_0)$. By \eqref{eq:4.3}, it suffices to study the oscillation of $w_{t_0,x_0,\rho}$ on $R_\rho(t_0,x_0)$. The oscillation of $w_{t_0,x_0,\rho}$ on $R_{\rho}(t_0,x_0)$ consists of those of $N^{(0)}$ and $E_{t_0,x_0,\rho}$. The oscillation of $E_{t_0,x_0,\rho}$ comes from those of $N^{(1)}$, $N^{(2)}$, $\hat v^{(1)}$ and $\hat u$.  Roughly speaking, the term in the square brackets in the definitions of $N^{(1)}$ is small, and so the oscillation of $N^{(1)}$ is small compared to the oscillation of $N^{(0)}$.  Also, in the definition of $N^{(2)}$, the heat kernel $G$ is small on the region of integration. The oscillations of $\hat v^{(1)}$ and $\hat u$ are small because $t_0^-$ is chosen far enough in the past of $t_0$ so that the heat kernel has the time to smooth the initial condition at time $t_0^-$, thanks to the growth bound $1+ \vert x \vert$ related to the stopping times $\tau_{K,2}$ and $\tau_{K,3}$. So altogether, we will see that $N^{(0)}$ is the term with dominant oscillation, and since it is Gaussian, we have precise estimates for it (see Proposition \ref{prop10.3}).
\vskip 12pt

\subsection{Oscillations of $N^{(0)}(t,x)$, $N^{(1)}_{t_0,x_0,\rho}(t,x)$ and $N^{(2)}_{t_0,x_0,\rho}(t,x)$}
\vskip 12pt


The random field $N^{(0)}(t,x)$ is a Gaussian process whose canonical metric is bounded by the metric $\Delta$ defined at the beginning of Section \ref{sec3}.  Thus Talagrand's analysis \cite{T1,T2} will apply to this case. To obtain a modulus of continuity for $N^{(0)}(t,x)$, we can use Corollary 
\ref{lem:osc-bound-rectangle} below with $\phi\equiv 1$, $\phi_1=1$, $S_0 = 0$, $S_1 = t_0 - \rho^4$, $T= t_0+\rho^4$, to obtain constants $C_0$ and $C_1$ such that for all $\lambda >0$,
\begin{equation*}
P\left(\text{osc}_{R_\rho(t_0,x_0)} (N^{(0)})>\lambda \rho\right)
 \leq C_0\exp\left(-C_1\lambda^2\right).
\end{equation*}
\vskip 12pt

\noindent{\em Oscillations of $N^{(1)}_{t_0,x_0,\rho}(t,x)$}
\vskip 12pt

\begin{lemma}\label{lem6.1}
Let $\beta$ be defined in \eqref{e4.1}. There are constants $C_0, C_1 \in \IR_+$ (which may depend on $K$ and the Lipschitz constant $L$) such that, for all $(t_0,x_0) \in [1,2] \times [0,1]$, $\rho \in \, ]0,\tfrac{1}{2}]$ and $\lambda >0$,
\begin{equation*}
P\left\{\text{\rm osc}_{R_\rho(t_0,x_0)}\left(N^{(1)}_{t_0,x_0,\rho}\right)>\lambda\right\}
\leq C_0\exp\left(-C_1\lambda^2\rho^{-2(1-\beta)(1-\delta)-2}\right).
\end{equation*}
\end{lemma}

\begin{proof}
Because $\sigma$ is Lipschitz with Lipschitz constant $L$ (by Assumption \ref{assump2.1}(a)), and since, for $(t,x) \in R_\rho(t_0,x_0)$, in the stochastic integral that defines $N^{(1)}_{t_0,x_0,\rho}(t,x)$, we have $t_0^-\leq s\leq t_0 + \rho^4$
and $|y- x_0|\leq L_1(\rho)$, we know from (\ref{eq:3.4}) that
\begin{align}
\label{eq:sigma-modulus}
\left|\sigma(u(s\wedge\tau_{K,1},y))-\sigma(u(t_0^-\wedge\tau_{K,1},x_0))\right|
&\leq L|u(s\wedge\tau_{K,1},y)-u(t_0^-\wedge\tau_{K,1},x_0)|  \\ \nonumber
&\leq LK\Delta(t_0+\rho^4-t_0^-,L_1(\rho))^{1-\delta}.   \\ \nonumber
&\leq LK \rho^{(1-\beta)(1-\delta)}.
\end{align}
It therefore follows from Corollary \ref{lem:osc-bound-rectangle} below, with
\begin{align*}
&\phi(r,z) = \sigma(u(r\wedge\tau_{K,1},z))-\sigma(u(t_0^-\wedge\tau_{K,1},x_0)),\qquad
\phi_1 = LK \rho^{(1-\beta)(1-\delta)}, \\
  & S_0 = t_0^-,\qquad S_1 = t_0 - \rho^4,\qquad T = t_0 + \rho^4,
\end{align*}
that for all $\lambda >0$,
\begin{equation*}
P\left\{\text{osc}_{R_\rho(t_0,x_0)} \big(N^{(1)}_{t_0,x_0,\rho}\big)
 >\lambda\right\}
\leq C_0\exp\left(-C_1 \lambda^2 \rho^{-2(1-\beta)(1-\delta)} (2\rho^4)^{-1/2} \right).
\end{equation*}
This proves the lemma.
\end{proof}

\noindent{\em Oscillations of $N^{(2)}_{t_0,x_0,\rho}(t,x)$}
\vskip 12pt

Recall that
\begin{equation*}
N^{(2)}_{t_0,x_0,\rho}(t,x)=\int_{t_0^-}^{t}\int_{[x_0-L_1,x_0+L_1]^c}G(t-s,x-y)
 \left[\sigma(u(s\wedge\tau_{K,1},y))
  -\sigma(u(t_0^-\wedge\tau_{K,1},x_0))\right]W(dy,ds).
\end{equation*}
Our goal for this subsection is to establish the following lemma.

\begin{lemma}\label{lem7.2}
Fix $\kappa>1$. There are constants $C_0, C_1 \in \IR_+$ (which may depend on $K$ and $\sigma_1$) such that, for all $(t_0,x_0) \in [1,2] \times [0,1]$, $\rho \in \, ]0,\tfrac{1}{2}]$ and $\lambda >0$,
\begin{equation}
\label{eq:bound-far-away}
P\left\{\sup_{(t,x)\in R_\rho(t_0,x_0)}|N^{(2)}_{t_0,x_0,\rho}(t,x)|>\lambda\right\}
\leq C_0\exp\left(-C_1\lambda^2\rho^{-2\kappa}\right)
\end{equation}
and
\begin{equation}
\label{eq:bound-oscillation-far-away}
P\left\{\text{\rm osc}_{R_\rho(t_0,x_0)} \left(N^{(2)}_{t_0,x_0,\rho}\right)>\lambda\right\}
\leq C_0\exp\left(-C_1\lambda^2\rho^{-2\kappa}\right),
\end{equation}
with possibly different constants $C_0$ and $C_1$.
\end{lemma}

\begin{proof} Since the oscillation of a function is bounded by twice its absolute
maximum, we see that (\ref{eq:bound-far-away}) implies \eqref{eq:bound-oscillation-far-away}. So we now prove \eqref{eq:bound-far-away}.

First, we split up $N^{(2)}$ as follows:
\begin{align*}
N^{(2)}_{t_0,x_0,\rho}(t,x)&= N^{(2a)}_{t_0,x_0,\rho}(t,x)+N^{(2b)}_{t_0,x_0,\rho}(t,x),
\end{align*}
where
\begin{align*}
N^{(2a)}_{t_0,x_0,\rho}(t,x) &= \int_{t_0^-}^{t}\int_{[x_0-L_1,x_0+L_1]^c}G(t-s,x-y)\sigma(u(s\wedge\tau_{K,1},y))\, W(dy,ds), \\
N^{(2b)}_{t_0,x_0,\rho}(t,x) &= \int_{t_0^-}^{t}\int_{[x_0-L_1,x_0+L_1]^c}G(t-s,x-y)\sigma(u(t_0^-\wedge\tau_{K,1},0))\, W(dy,ds).
\end{align*}
It suffices to show (\ref{eq:bound-far-away}) for $N^{(2)}$ replaced by $N^{(2a)}$ and $N^{(2b)}$.  By Assumption \ref{assump2.1}(b), the factor $\sigma(\cdots)$ in both $N^{(2a)}$ and $N^{(2b)}$ is bounded by $\sigma_1$, and that is the only information about $\sigma$ that we will use in our proof.  So we will only deal with $N^{(2a)}$, since the proof for $N^{(2b)}$ is identical.

The intuition is the following.
\vskip 12pt

\textbf{Step 1:} Since $R_\rho(t_0,x_0)$ is far away from
$[t_0^-,t_0 + \rho^4]\times[x_0-L_1,x_0+L_1]^c$, we expect $G(t-s,x-y)$ to be small for
$(t,x)\in R_\rho(t_0,x_0)$ and $(s,y)\in [t_0^-,t]\times[x_0-L_1,x_0+L_1]^c$.  This will lead to
a small value of $N^{(2a)}_{t_0,x_0,\rho}(t,x)$ when $(t,x)$ is fixed, with high
probability.
\vskip 12pt

\textbf{Step 2:} To go further and show that the supremum of
$|N^{(2a)}_{t_0,x_0,\rho}(t,x)|$ over $(t,x)\in R_\rho(t_0,x_0)$ is small, we divide $R_\rho(t_0,x_0)$ into even smaller subrectangles.  Ignoring the helpful fact that the domain of integration is far away from each of these subrectangles, we simply take the size of each
subrectangle to be so small that, by Corollary \ref{lem:osc-bound-rectangle}, the oscillation of $N^{(2a)}_{t_0,x_0,\rho}(t,x)$ over such a subrectangle is small with high probability.
\vskip 12pt

We begin with Step 2, which is a bit easier. Divide $R_\rho(t_0,x_0)$ into a union of nonoverlapping
subrectangles of dimensions $\rho^{4\kappa}\times\rho^{2\kappa}$, where
$\kappa>1$ is fixed in the statement of the lemma.  Since $R_\rho(t_0,x_0)$ has dimensions
$\rho^4\times\rho^2$, the number of subrectangles required is bounded by
$C\rho^{-6(\kappa-1)}$.  On each of these subrectangles $R'$, Corollary
\ref{lem:osc-bound-rectangle} with $\phi(r,z) = \sigma(u(r\wedge\tau_{K,1},z))$ in \eqref{eq:n-def}, $\phi_1 = \sigma_1$, $S_0 = t_0^-$, $T-S_1 = 2 \rho^{4\kappa}$, tells us that for all $\lambda >0$,
\begin{equation*}
P\left\{\text{osc}_{R'}\left(N^{(2a)}_{t_0,x_0,\rho}\right)>\lambda/2\right\}
 \leq C_0\exp\left(-C_1\lambda^2\rho^{-2\kappa}\right),
\end{equation*}
where the constant $C_1$ incorporates the constant $\sigma_1$.

Let $A_1(\rho,\kappa,\lambda)$ be the event that the oscillation of
$N^{(2a)}_{t_0,x_0,\rho}$ over each of these rectangles is less than or equal to
$\lambda/2$.  Then for all $\lambda >0$ and $\rho \in \, ]0,\tfrac{1}{2}]$,
\begin{align}
\label{eq:bound-osc-small-rectangles}
P\big(A_1(\rho,\kappa,\lambda)^c\big)&\leq C\rho^{6(1-\kappa)}\cdot
 C_0\exp\left(-C_1\lambda^2\rho^{-2\kappa}\right)  \\
&\leq C_0\exp\left(-C_1\lambda^2\rho^{-2\kappa}\right),  \nonumber
\end{align}
where $C_0,C_1$ may vary from line to line.

Now we turn to Step 1.
Let $R'$ be one of these subrectangles, and let $p'=(t',x')\in R'$ be a distinguished point.  We
wish to estimate $P\{|N^{(2a)}_{t_0,x_0,\rho}(p')|>\lambda/2\}$.  For $r\in[t_0^-,t']$, let
\begin{equation*}
M_r=\int_{t_0^-}^{r}\int_{[x_0-L_1,x_0+L_1]^c}G(t'-s,x'-y)\sigma(u(s\wedge\tau_{K,1},y))\, W(dy,ds)
\end{equation*}
and note that $(M_r)$ is an $(\mathcal{F}_r)$-martingale with quadratic variation
\begin{align}\nonumber
\langle M\rangle_r
&=\int_{t_0^-}^{r}\int_{[x_0-L_1,x_0+L_1]^c}G^2(t'-s,x'-y)\sigma^2(u(s\wedge\tau_{K,1},y))dyds \\ \nonumber
&\leq\sigma_1^2\int_{t_0^-}^{t'}\int_{[x_0-L_1,x_0+L_1]^c}G^2(t'-s,x'-y)dyds \\
&= \sigma_1^2\int_{0}^{t'-t_0^-}\int_{[x_0-L_1,x_0+L_1]^c} \frac{1}{2\pi r} \exp\left[-\frac{(x'-y)^2}{r} \right].
\label{e4.12}
\end{align}
Note that for $a>0$, $r\mapsto \tfrac{1}{r} \exp(-a^2/r)$ is increasing on the interval $]0,a^2]$, then decreasing.
Let $\alpha$ and $\beta$ be defined as in \eqref{e4.1}. For $r \in [0,t'-t_0^-]$ and $y \in [x_0-L_1,x_0+L_1]^c$,
\begin{equation*}
r\leq t_0+\rho^4-t_0^-\leq C\rho^{4(1-\alpha)} \qquad\text{and}\qquad
|x'-y|\geq \rho^{2(1-\beta)}.
\end{equation*}
Since $\beta>\alpha$, we replace $r$ by $t'-t_0^-$ in \eqref{e4.12} to see that
$$
   \langle M\rangle_r \leq \frac{\sigma_1^2}{2\pi} \int_{[x_0-L_1,x_0+L_1]^c} dy\, \exp\left[- \frac{(x'-y)^2}{t'-t_0^-}\right]
   \leq \frac{\sigma_1^2}{2\pi} \int_{[x_0-L_1,x_0+L_1]^c} dy\, \exp\left[- \frac{(x'-y)^2}{C \rho^{4(1-\alpha)}}\right].
$$
This integral is a sum of two integrals, over $]-\infty,x_0 - L_1]$ and $[x_0+L_1,+\infty[$. Both are bounded above by
\begin{align*}
  \int_{\rho^{2(1-\beta)}} dz \, \exp\left[ -\frac{z^2}{C \rho^{4(1-\alpha)}}\right] &= \int_{\rho^{2(\alpha-\beta)}} du \, \rho^{2(1-\alpha)} \exp\left[ -\frac{u^2}{C}\right]
  \leq C \rho^{2(1-\alpha)} \exp\left[ -\rho^{-4(\alpha-\beta)}\right]\\
  &\leq C \exp\left[ -\rho^{-4(\alpha-\beta)}\right].
\end{align*}
Finally, we obtain
\begin{equation*}
\langle M\rangle_r\leq C_0\exp\left(-C_1\rho^{-4(\beta-\alpha)}\right),\qquad r \in [t_0^-,t'].
\end{equation*}
Since $\beta>\alpha$, this exponential is small for small $\rho$.

Thus, $M_r$ is a time-changed Brownian motion with time scale bounded by our
bound on $\langle M\rangle_r$, and by the reflection principle for Brownian
motion,
\begin{align}
\nonumber
P\{|N^{(2a)}_{t_0,x_0,\rho}(t',x')|>\lambda/2\}
&\leq P\left\{\sup_{t_0^-\leq r\leq t'}|M_r|>\lambda/2\right\}   \\
&\leq 2P\left\{\sup_{0\leq t\leq \langle M\rangle_{t'}}B_t>\lambda/2\right\}   \nonumber\\
&\leq C_0\exp\Big(-C_2\lambda^2\exp\left(C_1\rho^{-4(\beta-\alpha)}\right)\Big) . \label{eq:est-dist-pt}
\end{align}

Let $A_2(\rho,\kappa,\lambda)$ be the event that for each subrectangle $R'$ and
for each distinguished point $p'\in R'$, we have
$|N^{(2a)}_{t_0,x_0,\rho}(p')|<\lambda/2$.  By
(\ref{eq:est-dist-pt}),
\begin{equation*}
P(A_2(\rho,\kappa,\lambda)^c)
 \leq C_0\rho^{-6(\kappa-1)}
  \exp\Big(-C_2\lambda^2\exp\left(C_1\rho^{-4(\beta-\alpha)}\right)\Big).
\end{equation*}
Therefore, using \eqref{eq:bound-osc-small-rectangles},
\begin{align}
\label{eq:bound-n-2}
P\left\{\sup_{(t,x)\in R_\rho(t_0,x_0)}|N^{(2a)}_{t_0,x_0,\rho}(t,x)| > \lambda\right\}
 &\leq P(A_1(\rho,\kappa,\lambda)^c) + P(A_2(\rho,\kappa,\lambda)^c)  \\
&\leq
 C\rho^{-6(\kappa-1)}\cdot C_0\exp\left(-C_1\lambda^2\rho^{-2\kappa}\right)
 \nonumber\\
&\qquad + C_0\rho^{-6(\kappa-1)}
  \exp\Big(-C_2\lambda^2\exp\left(C_1\rho^{-4(\beta-\alpha)}\right)\Big)
 \nonumber\\
&\leq C_0\exp\left(-C_2\lambda^2\rho^{-2\kappa}\right);  \nonumber
\end{align}
here, we have allowed the constants $C_0$, $C_2$ to vary from line to line.
\end{proof}

\subsection{Oscillations of $\hat{u}_{t_0,x_0,\rho}(t,x)$ and $\hat{v}^{(1)}_{t_0,x_0,\rho}(t,x)$}
Let $\hat{u}_{t_0,x_0,\rho}(t,x)$ and $\hat{v}^{(1)}_{t_0,x_0,\rho}(t,x)$ be as defined in \eqref{e4.4}.

\begin{lemma}\label{lem8.1}
Let $\alpha$ be as defined in \eqref{e4.1}.
There exists a constant $C$ (which may depend on $K$ and $L$) such that,
for all $(t_0,x_0) \in [1,2] \times [0,1]$, a.s., for all $\rho \in \, ]0,\tfrac{1}{2}]$,
\begin{equation*}
\text{\rm osc}_{R_\rho(t_0,x_0)}\big(\hat{u}_{t_0,x_0,\rho}\big) \leq C\rho^{2\alpha}\qquad \mbox{and}\qquad
\text{\rm osc}_{R_\rho(t_0,x_0)}\big(\hat{v}^{(1)}_{t_0,x_0,\rho}\big) \leq C\rho^{2\alpha}.
\end{equation*}
\end{lemma}


\begin{proof} The only difference between $\hat{u}_{t_0,x_0,\rho}(t,x)$ and $\hat{v}^{(1)}_{t_0,x_0,\rho}(t,x)$ is the initial data at $t=t_0^-$, which, in both cases, by the definitions of $\tau_{K,2}$ and $\tau_{K,3}$, have growth bounded by $K(1+|x|)$,
and this bound is all that we use in our proof.  Therefore we will only deal
with $\hat{u}_{t_0,x_0,\rho}(t,x)$, and leave $\hat{v}^{(1)}_{t_0,x_0,\rho}(t,x)$ to the reader.

As mentioned,
\begin{equation*}
|\hat{u}_{t_0,x_0,\rho}(t_0^-,x)|\leq K(1+|x|).
\end{equation*}
Recall from \eqref{e4.4} that for $(t,x) \in R_\rho(t_0,x_0)$,
\begin{equation*}
\hat{u}_{t_0,x_0,\rho}(t,x)
 = \int_{-\infty}^{\infty}G(t-t_0^-,x-z)\hat{u}_{t_0,x_0,\rho}(t_0^-,z)dz,
\end{equation*}
and so, using the inequalities $|z|+1\leq |x-z|+|x|+1\leq |x-z|+3$ since $|x| \leq 2$, and the standard inequality
\begin{equation*}
   \left| \frac{\partial G(t,x)}{\partial x} \right| \leq \frac{c}{\sqrt{t}} G(2t,x),
\end{equation*}
we obtain
\begin{align*}
\left|\frac{\partial\hat{u}_{t_0,x_0,\rho}(t,x)}{\partial x}\right|
&\leq C\int_{-\infty}^{\infty}
 (t-t_0^-)^{-1/2}\, G(2(t-t_0^-),x-z)\, (|z|+1)dz  \\
&\leq C\int_{-\infty}^{\infty}
  (t-t_0^-)^{-1/2}\, G(2(t-t_0^-),x-z)\, (|x-z|+3)dz  \\
&:= I_1+I_2,
\end{align*}
where
\begin{align*}
   I_1 &= C\int_{-\infty}^{\infty} |x-z|\, (t-t_0^-)^{-1/2}\, G(2(t-t_0^-),x-z)\, dz, \\
   I_2 &= 3C\, (t-t_0^-)^{-1/2}\int_{-\infty}^{\infty}  G(2(t-t_0^-),x-z)\, dz.
\end{align*}
Clearly, $I_2 = 3C (t-t_0^-)^{-1/2}$ and by change of variables, $I_1 \leq C'$, therefore, since $t-t_0^- \leq 1$,
\begin{equation}
\label{eq:w-deriv-x-est}
\left|\frac{\partial\hat{u}_{t_0,x_0,\rho}(t,x)}{\partial x}\right| \leq C(t-t_0^-)^{-1/2}.
\end{equation}

Next, again using $|z|+1\leq|x-z|+3$ since $|x| \leq 2$, and the standard inequality
\begin{equation*}
\left|\frac{\partial G(t,x)}{\partial t}\right| \leq \frac{c}{t} G(2t,x),
\end{equation*}
we find that
\begin{align*}
\left|\frac{\partial\hat{u}_{t_0,x_0,\rho}(t,x)}{\partial t}\right|
&\leq C\int_{-\infty}^{\infty}
    \frac{1}{t-t_0^-}\, G(2(t-t_0^-),x-z)(|x-z|+3)dz  \\
&:= I_3+I_4,
\end{align*}
where
\begin{align*}
   I_3 &= C\int_{-\infty}^{\infty} \frac{|x-z|}{t-t_0^-}\, G(2(t-t_0^-),x-z)\, dz, \\
   I_4&= 3C\, \frac{1}{t-t_0^-}\int_{-\infty}^{\infty}  G(2(t-t_0^-),x-z)\, dz.
\end{align*}
Clearly, $I_4 = 3C (t-t_0^-)^{-1}$, and again, a change of variables gives us $I_3 \leq C (t-t_0^-)^{-1/2}$, so we conclude that
\begin{equation}
\label{eq:w-deriv-t-est}
\left|\frac{\partial\hat{u}_{t_0,x_0,\rho}(t,x)}{\partial t}\right| \leq C(t-t_0^-)^{-1}.
\end{equation}

Now, since $t-t_0^-\geq\rho^{4(1-\alpha)}$,
we can bound the oscillation of $\hat u_{t_0,x_0,\rho}$ over $R_\rho(t_0,x_0)$ as follows. First, consider oscillation in the $x$-direction.  Let $I$ be a line segment contained in $R_\rho(t_0,x_0)$, consisting of points $(t,x)$ with $t$ fixed. The rectangle $R_\rho(t_0,x_0)$ has width $\rho^2$, so using (\ref{eq:w-deriv-x-est}), we get
\begin{equation*}
\text{osc}_I\big(\hat{u}_{t_0,x_0,\rho}\big)
\leq C\rho^2(t-t_0^-)^{-1/2}
\leq C\rho^2\left(\rho^{4(1-\alpha)}\right)^{-1/2}
= C\rho^{2\alpha}.
\end{equation*}
Second, consider oscillation in the $t$-direction.  Let $J$ be a line segment contained in $R_\rho(t_0,x_0)$, consisting of points $(t,x)$ with $x$ fixed. The rectangle $R_\rho(t_0,x_0)$ has height $\rho^4$, so using (\ref{eq:w-deriv-t-est}), we get
\begin{equation*}
\text{osc}_J\big(\hat{u}_{t_0,x_0,\rho}\big)
\leq C\rho^4(t-t_0^-)^{-1}
\leq C\rho^4\left(\rho^{4(1-\alpha)}\right)^{-1}
\leq C\rho^{4\alpha}.
\end{equation*}
Putting together these estimates establishes the conclusion of the lemma for $\hat u_{t_0,x_0,\rho}$.
\end{proof}



\section{Existence of rectangles with small oscillations}

For integers $q \geq 1$ and $\ell \geq 0$, set
\begin{equation*}
   \rpl = 2^{-q} q^{-\ell}  \qquad\mbox{and}\qquad \lsp = \left\lfloor\frac{q}{\log_2 q}\right\rfloor
\end{equation*}
(where $\log_2$ is the base $2$ logarithm), so that $r_{q,0} = 2^{-q}$ and $r_{q,\lsp} \geq 2^{-2q}$ (and is of the same order as $2^{-2q}$). Define
\begin{equation}\label{eq.fr}
   f(r) = r\left(\log_2\log_2 \tfrac{1}{r}\right)^{-1/6}.
\end{equation}
In \cite[Prop.2.3]{DMX}, we established a result for Gaussian random fields satisfying \cite[Assump 2.1]{DMX}. In \cite[Sect.7]{DMX}, we showed that this assumption was satisfied for systems of linear stochastic heat equations with i.i.d.~coefficients and vanishing initial condition (these last two assumptions are removed in \cite{DKLMX}). Since $N^{(0)}$ is the solution of such a system of linear stochastic heat equations, we restate here the result of \cite[Prop.2.3]{DMX} for $N^{(0)}$, in the form that we will need.

\begin{prop} There exist constants $\tilde K$ and $q_0$ with the following property: for all $q \geq q_0$ and for all $(t_0,x_0) \in [1,2]\times[0,1]$,
$$
  P\left\{\exists \ell \in \{0,\dots,\lsp\}: \mbox{\rm osc}_{R_{\rpl}(t_0,x_0)}(N^{(0)}) \leq \tilde K f(\rpl) \right\} \geq 1 - \exp\left(-\sqrt{q}\right).
$$
\label{prop10.3}
\end{prop}

\begin{proof} This proposition is essentially equivalent to \cite[Prop.2.3]{DMX} (applied to $N^{(0)}$), but since the notation is different, we explain how Proposition \ref{prop10.3} is obtained from the proof of \cite[Prop 2.3]{DMX}.

In the proof of \cite[Prop 2.3]{DMX}, we considered a sequence $r_\ell = r_0 U^{-2\ell}$, we set $\ell_0 = \lfloor \frac{\log_2 (1/r_0)}{2\log_2 U} \rfloor$ and we showed that
$$
   P\left\{\exists 1 \leq \ell \leq \ell_0: \mbox{\rm osc}_{R_{r_\ell}(t_0,x_0)}(N^{(0)}) \leq K_2 f(r_\ell) \right\} \geq 1 - \exp\left[ - \left(\log_2\frac{1}{r_0} \right)^{1/2}\right].
$$
Towards the end of the proof, $U$ was chosen by setting
$$
   U= \left(\log_2\frac{1}{r_0} \right)^{1/(2\beta)},
$$
where this $\beta$, defined in the proof using the H\"older exponents of the Gaussian random field, takes the value $\beta = 1$ in the case of $N^{(0)}$.

Here, we take $r_0$ of the form $r_0 = 2^{-q}$, so $U = q^{1/2}$, and the $r_\ell$, which now depend on $q$ and which we denote $\rpl$, take the value
$
   \rpl = 2^{-q} q^{-\ell},
$
and $\ell_0 = \lfloor \frac{q}{\log_2 q} \rfloor$, which we now denote $\ell_q$ in the statement of Proposition \ref{prop10.3}.
\end{proof}

In this section, we let $(t_0,x_0) \in R_0$, where
$$
   R_0 := [1,2] \times [0,1].
$$
We will establish the following theorem.

\begin{theorem} Consider the process $w_{t_0,x_0,r}$ defined in \eqref{eq:4.2} (with $\rho$ there replaced by $r$). Let $ \tilde K$ be the constant in Proposition \ref{prop10.3} and $f$ be the function defined in \eqref{eq.fr}. There is $\rho_0 \in \, ]0,\tfrac{1}{2}]$ with the following property. Given $0 < r_0 < \rho_0$, for all $(t_0,x_0) \in R_0$, we have
\begin{equation}\label{eq10.1}
   P\left\{\exists r \in [r_0^2,r_0]: \mbox{\rm osc}_{R_r(t_0,x_0)} (w_{t_0,x_0,r}) \leq 2 \sigma_1 \tilde K f(r) \right\} \geq 1 - 2 \exp\left[-\left(\log_2\frac{1}{r_0} \right)^{\frac12}\right]
\end{equation}
(we will only use this for $r_0$ of the form $2^{-q}$).
\label{thm10.2}
\end{theorem}

\begin{remark} The statement in this theorem should be compared with the statement for Gaussian processes in Proposition \ref{prop10.3}: notice the factor $2\sigma_1$ in front of $\tilde K$ and the factor $2$ in front of the exponential on the right-hand side.
\end{remark}

\begin{lemma} Let $E_{t_0,x_0,\rho}$ be as defined in \eqref{eq:4.2a}. There are constants $a >0$, $c_0 >0$ and $c_1>0$ such that, for all $(t_0,x_0) \in R_0$ and all sufficiently large $q$,
$$
   P\left\{\exists \ell \in \{0,\dots,\lsp\}: \mbox{\rm osc}_{R_{\rpl}(t_0,x_0)}(E_{t_0,x_0,\rpl}) \geq \sigma_1 \tilde K f(\rpl) \right\} \leq c_0 \exp\left(-c_1 2^{aq} \right).
$$
\label{lem10.4}
\end{lemma}

\begin{proof}
   Define the event
\begin{equation}\label{eq10.2}
   B_{q} = \left\{\exists \ell \in \{0,\dots,\lsp\}: \mbox{\rm osc}_{R_{\rpl}(t_0,x_0)}(E_{t_0,x_0,\rpl}) \geq \sigma_1 \tilde K f(\rpl) \right\}.
\end{equation}
Then
\begin{equation}\label{eq10.4}
B_{q} \subset \bigcup_{\ell=0}^{\lsp}\ \bigcup_{i=1}^4 B_{q,\ell}^{(i)},
\end{equation}
where
\begin{align*}
    B_{q,\ell}^{(1)} &= \left\{\mbox{\rm osc}_{R_{\rpl}(t_0,x_0)}(N^{(1)}_{t_0,x_0,\rpl}) \geq \tfrac{\sigma_1 \tilde K}{4} f(\rpl) \right\}, \\
    B_{q,\ell}^{(2)} &= \left\{\mbox{\rm osc}_{R_{\rpl}(t_0,x_0)}(N^{(2)}_{t_0,x_0,\rpl}) \geq \tfrac{\sigma_1 \tilde K}{4} f(\rpl) \right\}, \\
    B_{q,\ell}^{(3)} &= \left\{\mbox{\rm osc}_{R_{\rpl}(t_0,x_0)}(\hat v^{(1)}_{t_0,x_0,\rpl}) \geq \tfrac{ \tilde K}{4} f(\rpl) \right\}, \\
    B_{q,\ell}^{(4)} &= \left\{\mbox{\rm osc}_{R_{\rpl}(t_0,x_0)}(\hat u_{t_0,x_0,\rpl}) \geq \tfrac{\sigma_1 \tilde K}{4} f(\rpl) \right\}.
\end{align*}
By the definition of $\tau_{K,1}$ and Lemma \ref{lem6.1},
\begin{align*}
   P(B_{q,\ell}^{(1)}) &\leq C_0\exp\left[-C_1(\sigma_1 \tilde K f(\rpl))^2 \left(K\rpl^{-(1-\beta)(1-\delta)-1}\right)^2  \right]\\
   &= C_0 \exp\left[-C_1 \sigma_1^2 \tilde K^2 \frac{1}{(\log_2\log_2 \frac{1}{\rpl})^{2/6} } \tilde K^{-2} (\rpl)^{-2(1-\beta)(1-\delta)} \right].
\end{align*}
Therefore, for $a= 2(1-\beta)(1-\delta)$, we have
\begin{equation}\label{eq10.5}
   \sum_{\ell=0}^{\lsp} P(B_{q,\ell}^{(1)}) \leq c_0 \exp(-c_1 2^{a q}).
\end{equation}
We develop similar estimates for $B_{q,\ell}^{(2)}$, $B_{q,\ell}^{(3)}$ and $B_{q,\ell}^{(4)}$. According to \eqref{eq:bound-oscillation-far-away} in Lemma \ref{lem7.2},
$$
   P(B_{q,\ell}^{(2)}) \leq C_0 \exp\left[-C_1 \left(\tfrac{\sigma_1\tilde K}{4} f(\rpl)\right)^2 (\rpl)^{-2\kappa} \right].
$$
We take $\kappa = 2$, so that
\begin{equation}
   \sum_{\ell=0}^{\lsp} P(B_{q,\ell}^{(2)}) \leq C_0 \exp(-c_1 2^{2 q}).
\end{equation}

 For large $q$ and $\ell \in \{0,\dots,\lsp\}$, by Lemma \ref{lem8.1}, since $\alpha>\frac12$,
\begin{equation}
     P(B_{q,\ell}^{(3)}) = 0.
\end{equation}
Finally, for $B_{q,\ell}^{(4)}$, also by 
Lemma \ref{lem8.1}, for large $q$ and $\ell \in \{0,\dots,\lsp\}$,
\begin{equation}\label{eq10.8}
     P(B_{q,\ell}^{(4)}) = 0.
\end{equation}
Putting together \eqref{eq10.4}--\eqref{eq10.8} establishes Lemma \ref{lem10.4} with $a=2(1-\beta)(1-\delta)$.
\end{proof}


\noindent{\em Proof of Theorem \ref{thm10.2}.}  Define
$$
   A_{q,\ell} = \left\{\mbox{\rm osc}_{R_{\rpl}(t_0,x_0)}(N^{(0)}_{t_0,x_0,\rpl}) \leq \tilde K f(\rpl) \right\}.
$$
Consider the event $B_{q}$ defined in \eqref{eq10.2}. If we set $r_0 = 2^{-q}$, then the event in \eqref{eq10.1} contains the event
$$
   \left(\bigcup_{\ell = 0}^{\lsp} A_{q,\ell}\right) \setminus B_{q},
$$
whose probability, using Proposition \ref{prop10.3} and Lemma \ref{lem10.4}, is bounded below by
\begin{equation}
   1 - \exp(-\sqrt{q}) - P(B_{q}).
\end{equation}
From Lemma \ref{lem10.4}, we see that for $q$ large enough, this is bounded below by
$1 - 2\exp(-\sqrt{q})$, which proves the theorem.
\hfill $\Box$
\vskip 16pt

\section{Probability estimates via chaining}

Now we describe the chaining framework, which has been used in papers such as \cite{mueller1}. Let $R_1$ be a rectangle with side lengths no greater than $1$. By translating our coordinate system if necessary, we may assume that the
lower left hand corner of $R$ lies at the origin, so
\begin{equation*}
R_1=[0,a]\times[0,b],\qquad a,b \in [0,1].
\end{equation*}
We consider the grid
\begin{equation*}
\mathcal{G}_n=\{(k2^{-4n},\ell2^{-2n}):\ k,\ell \in \IN\}.
\end{equation*}
Note that the choice of exponents $-4n$, $-2n$ corresponds to the parabolic scaling: nearest neighbors in $\mathcal{G}_n$
have $\Delta$-distance $2^{-n}$.  Finally, let
\begin{equation*}
\mathcal{G}=\bigcup_{n=0}^\infty\mathcal{G}_n.
\end{equation*}
Also, we say that a closed rectangle $R$ is of type $n$ if each of the four edges
of $R$ is an interval whose endpoints are nearest neighbors in
$\mathcal{G}_n$. Two elements in a given rectangle of type $n$ are at most at $\Delta$-distance $2^{-n}$ of each other. We also say that a line segment (a step) is of type $n$ if its
endpoints are nearest neighbors in $\mathcal{G}_n$.  Finally, a path of type
$n$ is a path whose steps are line segments of type $n$.

\begin{lemma}
\label{lem:chaining}
Let $(\delta_n,\, n\in \IN)$ be a sequence of nonnegative numbers. Let $g:R_1\to\IR^d$ and suppose that for all nearest neighbor pairs
$p^{(1)}_n,p^{(2)}_n\in\mathcal{G}_n\cap R_1$, we have
\begin{equation*}
\big|g(p^{(1)}_n)-g(p^{(2)}_n)\big|\leq \delta_n, \qquad \text{for all } n\geq0.
\end{equation*}
If $p^{(1)},p^{(2)}\in\mathcal{G}\cap R_1$, then
\begin{equation*}
\big|g(p^{(1)})-g(p^{(2)})\big|\leq 40 \sum_{n=n_0}^{\infty}\delta_n,
\end{equation*}
where $n_0 = n_0(p^{(1)},p^{(2)})$ is the integer part of
$\log_2(1/\Delta(p^{(1)}-p^{(2)}))$ (so $n_0 \geq 0$).
\end{lemma}

Lemma \ref{lem:chaining} follows from the triangle inequality and from the
following lemma. 

\begin{lemma}
\label{lem:chaining2}
Let $p^{(1)},p^{(2)}\in\mathcal{G}\cap R_1$, and let
$n_0$ be the integer part of $\log_2(1/\Delta(p^{(1)}-p^{(2)}))$.
Then we can connect $p^{(1)}$ and $p^{(2)}$ by a path
consisting of line segments (steps) satisfying the following two conditions:
\begin{enumerate}
\item[(i)] each line segment is of type $n$ for some $n\geq n_0$;
\item[(ii)] for all $n\geq n_0$, there are at most 40 steps of type $n$.
\end{enumerate}
\end{lemma}

\begin{proof}
Item (i) is a requirement which enters into the proof of (ii).

Item (ii):  For $i=1,2$, let $R^{(n_0)}_i$ be the rectangle of type $n_0$ which
contains $p^{(i)}$.
First we claim that the rectangles $R^{(n_0)}_i$, $i=1,2$, are
either the same or they share a corner $q^{(1)}_{n_0} = q^{(2)}_{n_0}$ that also belongs to $R_1$.  We leave it to the reader to verify
this assertion, which is similar to the statement that if two real numbers $x$ and $y$ are such that $|x-y|\leq 1$, then either $x$ and $y$ lie in the same interval with integer endpoints, or they belong to two adjacent such intervals. 

In view of the above statement about rectangles
sharing a corner, we see that it is enough to show that 
we can connect any one of the corners of $R^{(n_0)}_1$ to $p^{(1)}$ using a path
with all corners in $\mathcal{G} \cap R_1$ and with at most 20 steps of type $n$ for each $n\geq n_0$.  Indeed, the same statement would hold for $i=2$, giving a path of $20+20=40$ steps of type $n$ altogether, using the shared corner as common starting point.

We can make a further reduction as follows.  
If $p^{(1)} \in \cG_{n_0}$, then it is one of the corners of $R^{(n_0)}_1$ and only two steps of type $n_0$ are needed to connect $p^{(1)}$ to $q_{n_0}$. Assume that $p^{(1)} \not\in \cG_{n_0}$ and that $n_1$ is the smallest integer $> n_0$ such that $p^{(1)} \in \cG_{n_1}$. For $n_0 < n < n_1$, let $R^{(n)}_1\subset R^{(n_0)}_1$ be a rectangle of type $n$ which contains
$p^{(1)}$.  Since this rectangle must intersect $R_1$, we can choose one of its
corners, denoted $q_n^{(1)}$, in $\mathcal{G}_{n} \cap R_1$.  We also require that $q_{n_0}^{(1)}$ is the shared corner mentioned above. For $n=n_1$, we let $q_{n_1}^{(1)} = p^{(1)}$.
It suffices to show that for $n_1 \geq n> n_0$, we can find a path
of type $n$ between $q_{n-1}^{(1)}$ and $q_{n}^{(1)}$, with at most 20 steps.
However, $R^{(n-1)}_1\cap\mathcal{G}_{n}$ consists of a $2^4\times2^2$
grid of points.  Given one of these points, and one of the corners of $R^{(n-1)}_1$, we can connect them by a path of type $n$ by taking at most $2^4=16$ steps of
type $n$ in the $t$-direction and at most $2^2=4$ steps of type $n$ in
the $x$-direction.  Altogether, this gives at most 20 steps of type $n$,
as we claimed.

Thus we have a path from $q_{n_0}^{(1)}$ to $q_{n_1}^{(1)}$ of the required
type.  To get the full path from $p^{(1)}$ to $p^{(2)}$, we put together the
two paths from $q^{(i)}_{n_0}$ to $p^{(i)}$, $i=1,2$, and we recall that
$q^{(1)}_{n_0} = q^{(2)}_{n_0}$.
\end{proof}

\noindent{\em Probability estimate for chaining}
\vskip 12pt

We use the notation of Lemmas \ref{lem:chaining} and \ref{lem:chaining2} .

\begin{lemma}
\label{lem:prob-est}
Let $n_{R_1}$ be the largest value of $n$ such that $\mathcal{G}_{n}\cap R_1$ is contained in a single rectangle of type $n$. For $n\geq n_{R_1}$, let $N(n)$ be the number of nearest neighbor pairs
in $\mathcal{G}_n\cap R_1$, so that for all $n \geq n_{R_1}$,
\begin{equation*}
N(n)\leq 2^{6n+1}+ 2^{4n} + 2^{2n}\leq 2^{6n+2}.
\end{equation*}
Let $(Y(t,x),\, (t,x) \in R_1)$ be an $\IR^d$-valued stochastic process. Let $(\delta_n,\, n\in \IN)\subset \IR_+$ and $(\varepsilon_n,\, n\in \IN)\subset \IR_+$ be two sequences of nonnegative numbers. Suppose that for all $n \geq n_{R_1}$ and for all nearest neighbor pairs $p^{(1)}_n,p^{(2)}_n\in\mathcal{G}_n\cap R_1$, we have
\begin{equation}
\label{eq:prob-est-epsilon-delta}
P\left\{\big|Y(p^{(1)}_n)-Y(p^{(2)}_n)\big|>\delta_n\right\} \leq \varepsilon_n.
\end{equation}
Let
\begin{equation}\label{eq9.2}
\varepsilon= 4\sum_{n=n_{R_1}}^{\infty}2^{6n}\varepsilon_n.
\end{equation}
Let $A$ be the event that for all $p^{(1)},p^{(2)}\in\mathcal{G}\cap R_1$, we have
\begin{equation*}
\big|Y(p^{(2)})-Y(p^{(1)})\big|\leq 40\sum_{n=n_0(p^{(1)},p^{(2)})}^{\infty}\delta_n,
\end{equation*}
where $n_0(p^{(1)},p^{(2)})$ $(\geq n_{R_1})$ is as in Lemma \ref{lem:chaining}.  Then
$
P(A^c)\leq\varepsilon.
$
\end{lemma}

\begin{proof}
Let $F_n$ be the event that for all nearest neighbor pairs
$p_n^{(1)},p_n^{(2)}\in\mathcal{G}_n\cap R_1$, we have
$|Y(p_n^{(1)})-Y(p_n^{(2)})|\leq\delta_n$, and let
$F=\cap_{n=n_{R_1}}^\infty F_n$.  By assumption
\eqref{eq:prob-est-epsilon-delta}, we have
\begin{equation*}
P(F^c)\leq \sum_{n=n_{R_1}}^{\infty} N(n) \varepsilon_n \leq 4 \sum_{n=n_{R_1}}^{\infty} 2^{6n}\varepsilon_n=\varepsilon.
\end{equation*}

Next we claim that on the set $F$, for all points
$p^{(1)},p^{(2)}\in\mathcal{G}\cap R_1$, we have
\begin{equation}
\label{eq:nearest-neighbor-estimate}
|Y(p^{(2)})-Y(p^{(1)})|\leq 40 \sum_{n=n_0(p^{(1)},p^{(2)})}^{\infty}\delta_n.
\end{equation}
Indeed, this follows from Lemma \ref{lem:chaining} with $g=Y$.
Therefore $A^c \subset F^c$, and this finishes the proof of Lemma
\ref{lem:prob-est}.
\end{proof}


The following estimates are standard \cite{khosh1}.  Here, $T>0$ and $0\leq s<t\leq T$.
\begin{lemma}
\label{lem:standard-Gaussian}
There exists a constant $C>0$ such that for all $0<s<t$, $x,y \in \IR$,
\begin{align*}
&\int_{0}^{t}\int_{-\infty}^{\infty}\left[G(t-r,x-z)-G(t-r,y-z)\right]^2dzdr
 \leq C|x-y|,  \\
&\int_{0}^{s}\int_{-\infty}^{\infty}\left[G(t-r,x-z)-G(s-r,x-z)\right]^2dzdr
 \leq C|t-s|^{1/2},  \\
&\int_{s}^{t}\int_{-\infty}^{\infty}\left[G(t-r,x-z)\right]^2dzdr
 \leq C|t-s|^{1/2}.
\end{align*}
\end{lemma}

We also give a probability estimate which we will use together with Lemma
\ref{lem:standard-Gaussian}.
\begin{lemma}
\label{lem:time-changed-bm}
There exist constants $C_0,C_1>0$ such that the following holds.
Suppose that $p^{(i)}=(t_i,x_i)\in[0,\infty)\times\IR$ for $i=1,2$.
Let $\phi(t,x)$ be jointly measurable and $(\mathcal{F}_t)$-adapted ($\IR^{d\times d}$-valued), and assume that there is $\phi_1 \in \IR_+$ such that
\begin{equation*}
\sup_{t,x}|\phi(t,x)|\leq \phi_1, \qquad \mbox{a.s.}
\end{equation*}
Define
\begin{equation*}
X_i = \int_{0}^{t_i}\int_{-\infty}^{\infty}G(t_i-s,x_i-y)\phi(s,y)W(dy,ds).
\end{equation*}
Then for $\lambda>0$,
\begin{equation*}
P\{|X_1-X_2|>\lambda\}
 \leq C_0\exp\left(-C_1\lambda^2\phi_1^{-2}\Delta(p^{(1)}-p^{(2)})^{-2}\right).
\end{equation*}
\end{lemma}

\begin{proof}
It is sufficient to prove this lemma in the case where $d=1$, so we assume that $d=1$ in the proof.
Note that by replacing $\phi$ by $\phi/\phi_1$ and $\lambda$ by
$\lambda/\phi_1$, we can reduce to the case $\phi_1=1$, so we assume from
now on that $\phi_1=1$ and so $\vert \phi(s,y)\vert \leq 1$.  By possibly changing indices, we may assume $t_1\leq t_2$.
Then
\begin{align*}
X_2-X_1
&= \int_{0}^{t_1}\int_{-\infty}^{\infty}
 \left[G(t_2-r,x_2-z)-G(t_1-r,x_2-z)\right]\phi(r,z)W(dz,dr)     \\
&\quad+ \int_{0}^{t_1}\int_{-\infty}^{\infty}
 \left[G(t_1-r,x_2-z)-G(t_1-r,x_1-z)\right]\phi(r,z)W(dz,dr)     \\
&\quad+ \int_{t_1}^{t_2}\int_{-\infty}^{\infty}G(t_2-r,x_2-z)\phi(r,z)W(dz,dr) \\
&=:  I_1+I_2+I_3.
\end{align*}
We analyze term $I_1$, leaving $I_2$ and $I_3$ to the reader using very similar
arguments.  For $0\leq t\leq t_1$, let
\begin{equation*}
M_t=\int_{0}^{t}\int_{-\infty}^{\infty}
 \left[G(t_2-r,x_2-z)-G(t_1-r,x_2-z)\right]\phi(r,z)W(dz,dr)
\end{equation*}
We note that $M_t$ is an $(\mathcal{F}_t)$-martingale with quadratic variation
\begin{align*}
\langle M\rangle_t
&=\int_{0}^{t}\int_{-\infty}^{\infty}
 \left[G(t_2-r,x_2-z)-G(t_1-r,x_2-z)\right]^2\phi(r,z)^2dzdr  \\
&\leq\int_{0}^{t}\int_{-\infty}^{\infty}
 \left[G(t_2-r,x_2-z)-G(t_1-r,x_2-z)\right]^2dzdr  \\
&\leq C\Delta(p^{(1)}-p^{(2)})^{2}
\end{align*}
by Lemma \ref{lem:standard-Gaussian}.  Thus $M_t$
is a time-changed Brownian motion with time scale $\tau(t)$ bounded by
\begin{equation*}
T_{\text{max}}=C\Delta(p^{(1)}-p^{(2)})^2
\end{equation*}
for $0\leq t\leq t_1$.  Noting that
$I_1=M_{t_1}$, we obtain from the reflection principle for Brownian motion and standard Gaussian estimates that
\begin{align*}
P\{|I_1|>\lambda/3\}
&\leq P\left\{\sup_{0\leq s\leq T_{\text{max}}}|B_s|>\lambda/3\right\}
\leq CP\{|B_{T_\text{max}}|>\lambda/3\}
\leq C_0\exp\left(-C_1\lambda^2T_\text{max}^{-1}\right)
\end{align*}
for appropriate constants $C_0,C_1>0$.

Similar estimates hold for $I_2$ and $I_3$.  Combining these estimates and
using the definition of $T_\text{max}$ finishes the proof of Lemma
\ref{lem:time-changed-bm}.
\end{proof}

\noindent{\em Probability bounds for the modulus of continuity}
\vskip 12pt

In this section, we combine Lemmas \ref{lem:prob-est} and
\ref{lem:time-changed-bm} to get the probability bound in Lemma \ref{lem:modulus-bound} below.
For this section, let
\begin{equation}
\label{eq:n-def-1}
N^{(3)}(t,x)=N^{(3)}(t,x,\phi)
 =\int_{0}^{t}\int_{-\infty}^{\infty}G(t-r,x-z)\phi(r,z)W(dz,dr),
\end{equation}
where $\phi(r,z)$ is a jointly measurable and $(\mathcal{F}_t)$-adapted $\IR^{d \times d}$-valued process, and for some $\phi_1 \in \IR_+$,
\begin{equation}
\label{eq:n-def-1a}
\sup_{r,z}|\phi(r,z)|\leq\phi_1,\qquad \mbox{a.s.}
\end{equation}
We will be using the jointly continuous version of $N^{(3)}$ (which exists by Lemma \ref{lem:standard-Gaussian} and \cite[Propositions 4.3 \& 4.4]{chen_dal}).

\begin{lemma}
\label{lem:modulus-bound}
Fix $\lambda_0>0$. There exist constants $C_0$ and $C_1$ such that the following holds.  For $\rho \in \, ]0,1]$ and $\lambda \geq \lambda_0$, for each rectangle $R\subset R_0 = [1,2] \times [0,1]$ of dimensions $\rho^4\times\rho^2$,  let $A_\lambda(R)$ be the event that for all $p^{(1)},p^{(2)}\in R$,
\begin{equation*}
\big|N^{(3)}(p^{(1)})-N^{(3)}(p^{(2)})\big|
 \leq \lambda\Delta(p^{(1)}-p^{(2)})
 \log_+\left(1/\Delta(p^{(1)}-p^{(2)})\right),
\end{equation*}
where for $\gamma>0$, $\log_+(\gamma) := \max(1,\log_2(\gamma))$. Then
\begin{equation*}
P(A_\lambda(R)^c)\leq C_0\exp\left(-C_1\lambda^2\phi_1^{-2}\log_+^2(1/\rho)\right).
\end{equation*}
\end{lemma}

\begin{proof}
As in the proof of Lemma \ref{lem:time-changed-bm}, first note that by replacing $\phi$ by $\phi/\phi_1$ and $\lambda$ by $\lambda/\phi_1$, we can reduce to the case where $|\phi(r,z)|\leq1$, so we assume that
that $\phi_1=1$.  

Now let $n_1\in \IN$ be such that $2^{-n_1-1} < \rho \leq 2^{-n_1}$, and for $n\in \IN$, set
\begin{align*}
\delta_n=c\lambda (n+1)2^{-n},  \qquad
\varepsilon_n=C_0\exp\left(-C_1c^2\lambda^2n^2\right),
\end{align*}
with $c>0$ to be defined later, and $C_0$ and $C_1$ are the constants from Lemma \ref{lem:time-changed-bm}. We want to use Lemma \ref{eq:nearest-neighbor-estimate}, but this lemma was stated for rectangles with one corner at the origin, and since this is not the case for $R$, we are going to shift the grid $\cG$. Suppose that $R= [s_0,s_0+\rho^4]\times [y_0,y_0+\rho^2]$, where $p_0=(s_0,y_0) \in [1,2]\times [0,1]$. In the statement of Lemma \ref{eq:nearest-neighbor-estimate}, we replace $R_1$ by $R$, $\cG_n$ by $p_0 + \cG_n$ and $\cG$ by $p_0 + \cG$, without affecting the validity of the statement in Lemma \ref{eq:nearest-neighbor-estimate}.

Next, we use Lemma \ref{lem:time-changed-bm} in order to check
(\ref{eq:prob-est-epsilon-delta}), for $p_n^{(i)}\in p_0 +\mathcal{G}_n$ for $i=1,2$.
Lemma \ref{lem:time-changed-bm} yields
\begin{align*}
P\left\{\big|N^{(3)}(p_n^{(1)})-N^{(3)}(p_n^{(2)})\big|>\delta_n\right\}
&\leq C_0\exp\left(-C_1\delta_n^2\Delta(p_n^{(1)}-p_n^{(2)})^{-2}\right).
\end{align*}
For $p^{(i)}_n\in (p_0 + \mathcal{G}_n)\cap R$ nearest neighbor pairs in $p_0 +\mathcal{G}_n$, we have
\begin{equation*}
\Delta(p_n^{(1)}-p_n^{(2)})= 2^{-n},
\end{equation*}
and so, using the definition of $\delta_n$,
\begin{align*}
P\left\{\big|N^{(3)}(p_n^{(1)})-N^{(3)}(p_n^{(2)})\big|>\delta_n\right\}
\leq C_0\exp\left(-C_1\delta_n^22^{2n}\right)
= C_0\exp\left(-C_1c^2\lambda^2 (n+1)^2\right).
\end{align*}

Now we use Lemma \ref{eq:nearest-neighbor-estimate} with the shifted grid to complete the
proof.  In that lemma, let $Y(p)=N^{(3)}(p)$ and note that we have
verified condition \eqref{eq:prob-est-epsilon-delta} with $\varepsilon_n = C_0\exp\left(-C_1c^2\lambda^2n^2\right)$.

Also, for $p^{(1)},p^{(2)}\in (p_0+\mathcal{G}) \cap R$ and $n_0=n_0(p^{(1)},p^{(2)})$ (defined in Lemma \ref{lem:chaining}), we compute
\begin{align*}
40\sum_{m=n_0}^{\infty}\delta_m
\leq C_5\, c\lambda\, (n_0+1)\, 2^{-n_0}
\leq \lambda\Delta(p^{(1)}-p^{(2)})\log_+ \left(1/\Delta(p^{(1)}-p^{(2)})\right)
\end{align*}
for some constant $C_5$ and for $c$ small enough.

Continuing with formula \eqref{eq9.2}, we define (with $n_{R_1}=n_1$)
\begin{align*}
\varepsilon&=2\sum_{n=n_1}^{\infty}2^{6n}\varepsilon_n
=2\sum_{n=n_1}^{\infty}2^{6n}C_0\exp\left(-C_1c^2\lambda^2 (n+1)^2\right)  \\
&\leq C_0\exp\left(-C_1c^2\lambda^2 (n_1+1)^2\right)
\end{align*}
for $\lambda\geq \lambda_0$ and as usual, $C_0,C_1$ changing from line to line.

Since $n_1+1 > \log_+(1/\rho)$, the conclusion of Lemma \ref{lem:prob-est} directly implies the
conclusion of Lemma \ref{lem:modulus-bound}.
\end{proof}


\noindent{\em Probability bounds for oscillation over a rectangle}
\vskip 12pt

Now we prove an estimate similar to Lemma \ref{lem:modulus-bound}, but
for the modulus of continuity over a rectangle.

Assume that
\begin{gather*}
0\leq S_0\leq S_1\leq t\leq T ,\qquad
x\in\IR,
\end{gather*}
and let
\begin{equation}
\label{eq:n-def}
N^{(4)}(t,x)=N^{(4)}(t,x,\phi,S_0,S_1)
 =\int_{S_0}^{t}\int_{-\infty}^{\infty}G(t-r,x-z)\phi(r,z)W(dz,dr),
\end{equation}
where $\phi$ is a jointly measurable and $(\mathcal{F}_t)$-adapted $\IR^{d \times d}$-valued process, and for some $\phi_1 \in \IR_+$,
\begin{equation}\label{e6.7}
\sup_{t,x}|\phi(t,x)| \leq \phi_1\qquad \mbox{a.s.}
\end{equation}
By Lemma \ref{lem:standard-Gaussian} and \cite[Propositions 4.3 \& 4.4]{chen_dal}, $N^{(4)}$ has a continuous version.

\begin{cor}
\label{lem:osc-bound-rectangle}
Let $N^{(4)}(t,x)$ be as in (\ref{eq:n-def}). There exist constants $C_0,C_1>0$ such that for $0\leq S_0\leq S_1\leq T$ and $y_0 \in \IR$, letting
\begin{equation*}
R_1=[S_1,T]\times\left[y_0,y_0+(T-S_1)^{1/2}\right]
\end{equation*}
we have for all $\lambda >0$,
\begin{equation}
\label{eq:prob-osc-est}
P\big\{{\rm osc}_{R_1}(N^{(4)})>\lambda\big\}
 \leq C_0\exp\left(-C_1\lambda^2\phi_1^{-2}(T-S_1)^{-1/2}\right).
\end{equation}
\end{cor}

\begin{proof}
First, note that by translation of the time and space variables, which preserves the space-time white noise and property \eqref{e6.7}, it suffices to consider the case $S_0=0$ and $y_0 = 0$, so we will assume $S_0=0$ and $y_0 = 0$ from now on.

Secondly, by considering $N^{(4)}(t,x,\phi/\phi_1,0,S_1)$ and noting that
\begin{equation*}
N^{(4)}(t,x,\phi/\phi_1,0,S_1)=\phi_1^{-1}N^{(4)}(t,x,\phi,0,S_1),
\end{equation*}
we can remove the dependence on $\phi_1$ and assume that $\phi_1 = 1$.
Thirdly, we will remove the dependence on $T-S_1$ by scaling: let
\begin{align*}
\tilde{N}(t,x)&=\phi_1^{-1}(T-S_1)^{-1/4}N^{(4)}\left(t(T-S_1),x(T-S_1)^{1/2}\right),  \\
\tilde{\phi}(t,x)&=\phi_1^{-1}\phi(t(T-S_1),x(T-S_1)^{1/2}),
\end{align*}
and note that $\tilde{N}(t,x)$ is of the form (\ref{eq:n-def}) with the following
modifications:
\begin{equation*}
\tilde{N}(t,x)
 =\int_{0}^{t}\int_{-\infty}^{\infty}G(t-r,x-z)\tilde{\phi}(r,z)\tilde{W}(dzdr),
\end{equation*}
where $\tilde{W}$ is another space-time white noise and
$|\tilde{\phi}|\leq1$.  With these transformations, the rectangle $R_1$ is replaced by
\begin{equation*}
R_2=[a,a+1]\times[0,1],
\end{equation*}
with $a=S_1/(T-S_1)$.  In particular,
\begin{align}
\label{eq:osc-unit-rectangle}
P\left\{{\rm osc}_{R_1}(N^{(4)})>\lambda\right\}
&=P\left\{{\rm osc}_{R_2}(\phi_1(T-S_1)^{1/4}\tilde{N})>\lambda\right\}   \\
&=P\left\{{\rm osc}_{R_2}(\tilde{N})>\lambda\phi_1^{-1}(T-S_1)^{-1/4}\right\}.  \nonumber
\end{align}

Thus it suffices to prove Corollary \ref{lem:osc-bound-rectangle} for $\tilde{N}$ and $S_0=0$, $y_0 = 0$, $T=S_1+1$. With these changes, since $S_0=0$, we can replace $\tilde N$ by $N^{(3)}$, as defined in \eqref{eq:n-def-1}.

We now reduce Corollary \ref{lem:osc-bound-rectangle} to Lemma
\ref{lem:modulus-bound}.  In Lemma \ref{lem:modulus-bound}, let
$R=R_2$, which is a $1\times1$ rectangle, so $\rho=1$. Let $\lambda_0 = 1$. Observe that $\vert x \log_+(1/x) \vert \leq 1$ for $x \in \,]0,1]$, and for points $p^{(1)},p^{(2)}\in R_2$, we have $\Delta(p^{(1)}-p^{(2)})\leq 1$, so
on the event $A_\lambda(R_2)$ in Lemma \ref{lem:modulus-bound}, we have
\begin{equation*}
\big|N^{(3)}(p^{(1)})-N^{(3)}(p^{(2)})\big|
 \leq \lambda
\end{equation*}
for all points $p^{(1)},p^{(2)}\in R_2$ and thus
$\text{osc}_{R_2}(N^{(3)})\leq \lambda$.
Therefore, if $A'$ is the event that $\text{osc}_{R_2}(N^{(3)})\leq \lambda$, then
we have $A_\lambda(R_2)\subset A'$ and also $(A')^c\subset A_\lambda(R_2)^c$.  
Therefore, Corollary \ref{lem:osc-bound-rectangle} will
follow if we can show that for all $\lambda>0$,
\begin{equation}
\label{eq:to-show-cor}
P(A_\lambda(R_2)^c)\leq C_0\exp\left(-C_1\lambda^2\right),
\end{equation}
where we recall that we are in the case $\phi_1=1$ and $\rho=T-S_1=1$.
However, the conclusion of Lemma \ref{lem:modulus-bound} gives
\eqref{eq:to-show-cor} for $\lambda\geq \lambda_0 = 1$.  To deal with
$0<\lambda<\lambda_0$, it suffices to increase $C_0$ if necessary, so
that $C_0\exp\left(-C_1\lambda_0^2\right)\geq1$.

This establishes \eqref{eq:to-show-cor} and finishes the proof of
Corollary \ref{lem:osc-bound-rectangle}.


\end{proof}

\noindent{\em Growth of $N^{(3)}(t,x)$ as  $|x|\to \infty$}
\vskip 12pt

Let $N^{(3)}(t,x)$ be the jointly continuous version of the process defined in \eqref{eq:n-def-1}.
\begin{lemma}
\label{lem6.8}
Fix $T>0$. There exists an almost surely finite
random variable $Z$ such that with probability one, for all $x \in \IR$,
\begin{equation*}
\sup_{t\leq T}\big|N^{(3)}(t,x)\big|\leq Z(|x|+1).
\end{equation*}
\end{lemma}

\begin{proof}
We split up $[0,T]\times\IR$ into unit blocks $B_n=[0,T]\times[n,n+1]$ and let 
\begin{equation*}
N_n=\sup_{(t,x)\in B_n}|N^{(3)}(t,x)|.
\end{equation*}
Since $N^{(3)}(0,x) \equiv 0$, and $B_n$ is the union of at most $[T]+1$ squares, to each of which Corollary \ref{lem:osc-bound-rectangle} applies, 
for each $n\in\mathbb{Z}$, we have
\begin{equation*}
P\{N_n>\lambda(n+1)\}\leq C_0\exp\left(-C_1\lambda^2(n+1)^2\right).
\end{equation*}
Here we have incorporated $\phi_1$ into $C_1$.
By the Borel-Cantelli lemma, there exists an almost surely finite
random variable $Z$ such that with probability one, for all $x \in \IR$,
\begin{equation*}
\sup_{t\leq T}\left|N^{(3)}(t,x)\right|\leq Z(|x|+1).
\end{equation*}
\end{proof}

\section{Establishing polarity of almost all points for $d \geq 6$}\label{sec7}

The goal of this section is to prove Theorem \ref{thm11.4}. We want to study the range of $\tilde u$, which takes values in $\IR^d$ ($d \geq 6$), as $(t,x)$ varies in the time-space rectangle $R_0= [1,2] \times [0,1]$.

Let $w_{s,y,r}$ be as defined in \eqref{eq:4.2} with $t_0$, $x_0$, $\rho$ there replaced respectively by $s$, $y$, $r$. Let $\tilde K$ be the constant in Proposition \ref{prop10.3}.  For $q \geq 1$, consider the random set
$$
   G_q =\left\{(s,y) \in R_0: \exists r \in [2^{-2q},2^{-q}[ \mbox{ with } \mbox{osc}_{R_r(s,y)}(w_{s,y,r}) \leq 2 \sigma_1\tilde K f(r) \right\},
$$
where the function $f$ is defined in \eqref{eq.fr}, and the event
$$
   \Omega_{q,1} = \left\{\lambda_2(G_q) \geq \lambda_2(R_0) \left(1-\exp(-\sqrt{q}/4) \right) \right\}
$$
(here, $\lambda_2$ denotes Lebesgue measure on $\IR^2$, so $\lambda_2(R_0) = 1$). Then
\begin{align*}
   (\Omega_{q,1})^c &=  \left\{\lambda_2(G_q) < \lambda_2(R_0) \left(1-\exp(-\sqrt{q}/4) \right) \right\} \\
   &= \left\{\lambda_2(R_0\setminus G_q) > \lambda_2(R_0) \exp(-\sqrt{q}/4)  \right\}.
\end{align*}
By Markov's inequality,
\begin{equation}\label{eq11.1}
 P((\Omega_{q,1})^c) \leq \frac{E[\lambda_2(R_0\setminus G_q)]}{\lambda_2(R_0) \exp(-\sqrt{q}/4)}.
\end{equation}
The numerator is equal to
$$
   E\left[\int_{R_0} 1_{R_0\setminus G_q}(s,y) \, dsdy \right] = \int_{R_0} P\{(s,y) \in R_0\setminus G_q \}\, dsdy.
$$
By definition of $G_q$ and Theorem \ref{thm10.2}, for all $(s,y) \in R_0$,
$$
   P\{(s,y) \not\in G_q\} \leq 2 \exp\left[-\left(\log_2 \frac{1}{2^{-q}} \right)^{\frac12} \right] = 2 \exp(-\sqrt{q}),
$$
therefore, by \eqref{eq11.1},
$$
   P((\Omega_{q,1})^c) \leq 2 \exp\left[-\frac{3}{4} \sqrt{q} \right].
$$
In particular,
\begin{equation}\label{eq11.2}
   \sum_{q=1}^{\infty} P((\Omega_{q,1})^c) < +\infty.
\end{equation}

On $\Omega_{q,1}$, for each $(s,y) \in G_q$, there exists $r \in [2^{-2q},2^{-q}]$ such that
\begin{equation}\label{eq11.3}
   \mbox{osc}_{R_r(s,y)}(w_{s,y,r}) \leq 2 \sigma_1 \tilde K f(r).
\end{equation}

Define an ``anisotropic dyadic rectangle" of order $\ell$ as a rectangle in $\IR_+ \times \IR$ of the form
$$
   [m_1 2^{-4\ell},\, (m_1+1) 2^{-4\ell}] \times [m_2 2^{-2\ell},\, (m_2+1) 2^{-2\ell}],
$$
where $m_1, m_2 \in \IN$. For $(s,y) \in R_0$, let $Q_\ell(s,y)$ denote the anisotropic dyadic rectangle of order $\ell$ that contains $(s,y)$. This rectangle is called ``good" if
\begin{equation}\label{eq11.4}
   \mbox{osc}_{Q_\ell(s,y)} (\tilde u) \leq d_\ell,
\end{equation}
where
$$
   d_\ell = 8 \sigma_1 \tilde K f(2^{-\ell}).
$$
By \eqref{eq11.3}, when $\Omega_{q,1} \cap \{\tau_{K,2} \wedge \tau_{K,3} = T_0 \}$ occurs (so that, by \eqref{eq:4.3}, all the $w_{s,y,r}$ are equal to $\tilde u$), we can find a family $\cH_{q,1}$ of non-overlapping good dyadic rectangles, each of some order $\ell \in [q,2q]$, that covers $G_q$. This family is determined by the random field $\tilde u$.

Let $\cH_{q,2}$ be the family of non-overlapping dyadic rectangles of order $2q$ that meet $R_0$ but are not contained in any of rectangle of $\cH_{q,1}$. For $q$ large enough, these rectangles are contained in $[1,2]\times[0,1]$. Therefore, when $\Omega_{q,1} \cap \{\tau_{K,2} \wedge \tau_{K,3} = T_0 \}$ occurs, their number  is at most $N_q$, where
$$
    N_q\, 2^{-2q\cdot 6} \leq \lambda_2(R) \exp(-\sqrt{q}/4),
$$
so
\begin{equation}\label{eq11.5}
   N_q \leq C 2^{12 q} \exp(-\sqrt{q}/4),
\end{equation}
where $C$ does not depend on $q$.

Let $\Omega_{q,2}$ be the event ``for all dyadic rectangles $R$ of order $2q$ that meet $R_0$, the inequality
\begin{equation}\label{eq11.6}
   \mbox{osc}_{R} (\tilde u) \leq K_2 2^{-2q} q
\end{equation}
holds." The next statement is a consequence of Corollary \ref{lem:osc-bound-rectangle}.

\begin{lemma}\label{lem11.1}
There are constants $c_1$, $c_2$ such that, for $K_2$ large enough, for all $q\geq 1$, we have
$$
   P(\Omega_{q,2}^c) \leq c_1 \exp[-c_2 K_2^2 q^2].
$$
\end{lemma}

\begin{proof} Let $R$ be a dyadic rectangle of order $2q$ that meets $R_0$. Consider the event
$$
   H(R) = \left\{ \mbox{\rm osc}_R (\tilde u) \leq K_2 2^{-2q} q \right\}.
$$
Recall from \eqref{e3.5} that $\tilde u(t,x) = I(t,x) + N(t,x)$, where $I(t,x)$ is a deterministic integral and $N(t,x)$ is a stochastic integral of the same form as the process $N^{(4)}(t,x)$ in \eqref{eq:n-def}, with the bound $\phi_1 = \sigma_1$ given by Assumption \ref{assump2.1}(b). We note that on $R_0 = [1,2]\times [0,1]$, $(t,x) \mapsto I(t,x)$ is $C^\infty$, hence Lipschitz continuous with some Lipschitz constant $\tilde L$. Choose $K_2 \geq 2 \tilde L$. Then for $q \geq 1$, $\mbox{\rm osc}_R(I) \leq \tilde L \, 2^{-4q+1} \leq \tilde L 2^{-2q}$, therefore,
$$
   H(r)^c = \{\mbox{\rm osc}_R(\tilde u) > K_2 2^{-2q} q \} = \left\{\mbox{\rm osc}_R(N) \geq \tfrac{K_2}{2} 2^{-2q} q\right\}.
$$
By Corollary \ref{lem:osc-bound-rectangle} applied to $N(t,x)$ with $T - S_1 = 2^{-8q}$, we see that
\begin{align*}
   P(H(r)^c) &\leq P\left\{\mbox{\rm osc}_R(N) \geq \tfrac{K_2}{2} 2^{-2q} q\right\} \\
 &\leq C_0 \exp\left[-C_1 \left(\tfrac{K_2}{2} 2^{-2q} q \right)^2 \sigma_1^{-2} (2^{-8q})^{-1/2} \right] \\
 &= C_0 \exp\left[-\tilde C_1 K_2^2 q^2 \right].
\end{align*}
It follows that
$$
   P(\Omega_{q,2}^c) \leq 2^{12 q} C_0 \exp\left[-\tilde C_1 K_2^2 q^2 \right] \leq c_1 \exp[-c_2 K_2^2 q^2]
$$
for $K_2$ large enough. This proves Lemma \ref{lem11.1}.
\end{proof}

We continue working towards the proof of Theorem \ref{thm11.4}: We choose $K_2$ large enough so that
\begin{equation}\label{eq11.6a}
    \sum_{q = 1}^{\infty} P((\Omega_{q,2})^c) < +\infty:
\end{equation}
this is possible by Lemma \ref{lem11.1}.

Set $\cH_q = \cH_{q,1} \cap \cH_{q,2}$. This is a non-overlapping cover of $R_0$ (because of how dyadic rectangles fit together). Set
\begin{equation*}
\begin{array}{lll}
   r_A &= d_\ell &\mbox{if } A \in \cH_{q,1}\mbox{ and }A \mbox{ is of order } \ell \in [q,2q],\\
   r_A &= K_2\, 2^{-2q} q & \mbox{if } A \in \cH_{q,2}.
   \end{array}
\end{equation*}
Define
$$
   \Omega_q = \Omega_{q,1} \cap \Omega_{q,2}.
$$

\begin{lemma}\label{lem11.2} For $x>0$, let
\begin{equation}\label{eq11.7}
   \zeta(x) = x^6 \log_2\log_2\frac{1}{x}.
\end{equation}
For $q$ large enough, if $\Omega_{q,1} \cap \{\tau_{K,2} \wedge \tau_{K,3} = T_0 \}$ occurs, then
$$
   \sum_{A \in \cH_q} \zeta(r_A) \leq K \lambda_2(R_0).
$$
\end{lemma}

\begin{proof} For $A \in \cH_{q,1}$, if $A$ is of order $\ell \in [q,2q]$, then
$$
   \zeta(r_A) \leq K \left(\frac{2^{-\ell}}{(\log_2 \ell)^{1/6}} \right)^6 \log_2\log_2 2^\ell = K 2^{-6 \ell},
$$
and the right-hand side is the volume of an anisotropic rectangle of order $\ell$.

For $q$ large enough and for $A \in \cH_{q,2}$,
$$
   \zeta(r_A) \leq K (2^{-2q} q)^6 \log_2(2q),
$$
hence by \eqref{eq11.5}, on $\Omega_{q,1} \cap \{\tau_{K,2} \wedge \tau_{K,3} = T_0 \}$,
\begin{align*}
    \sum_{A \in \cH_{q,2}} \zeta(r_A) &\leq K (2^{-2q} q)^6 \log_2(2q) \cdot C 2^{12q} \exp(-\sqrt{q}/4) \\
    &= C K q^6 \log_2(2q) \exp(-\sqrt{q}/4).
\end{align*}
Therefore, since rectangles in $\cH_{q,1}$ are non-overlapping and contained in $R_0 = [1,2]\times [0,1]$,
$$
   \sum_{A \in \cH_q} \zeta(r_A) \leq K \lambda_2(R_0) + C K q^6 \log_2(2q) \exp(-\sqrt{q}/4).
$$
The first term on the right-hand side does not depend on $q$, while the second has limit $0$ as $q\to \infty$, so Lemma \ref{lem11.2} is proved.
\end{proof}

For each $A \in \cH_q$, we pick a distinguished point $(s_A,y_A) \in A$ (say, the lower left corner). Let $B_A$ be the Euclidean ball in $\IR^d$ centered at $\tilde u(s_A,y_A)$ with radius $r_A$.

\begin{lemma}\label{lem11.3}
Let $\cF_q$ be the family of balls $(B_A,\, A \in \cH_q)$. For $q$ large enough, on $\Omega_{q} \cap \{\tau_{K,2} \wedge \tau_{K,3} = T_0 \}$, $\cF_q$ covers the random set
$$
   \tilde M = \{\tilde u(s,y): (s,y) \in R_0 \}.
$$
Notice that $\tilde M \subset \IR^d$ is the range of $\tilde u$ as $(s,y)$ varies in $R_0= [1,2]\times [0,1]$.
\end{lemma}

\begin{proof} Fix $z \in \tilde M$. By definition, there is $(s,y) \in R_0$ such that $z = \tilde u(s,y)$. Since $\cH_q$ is a cover of $R_0$, the point $(s,y)$ belongs to some rectangle $A$ of $\cH_q$.

Consider first the case where $A \in \cH_{q,1}$. Suppose that $A$ is of order $\ell \in [q,2q]$. By \eqref{eq11.4},
$$
   \vert \tilde u(s,y) - \tilde u(s_A,y_A) \vert \leq d_\ell, \qquad\mbox{that is, }\qquad \vert z - \tilde u(s_A,y_A) \vert \leq r_A.
$$
This means that $z \in B_A$.

Now consider that case where $A \in \cH_{q,2}$. Then on $\Omega_{q,2}$, by \eqref{eq11.6},
$$
   \vert z - \tilde u(s_A,y_A)\vert = \vert \tilde u(s,y) - \tilde u(s_A,y_A)\vert \leq K_2\, 2^{-2q} q = r_A,
$$
so $z \in B_A$. The lemma is proved.
\end{proof}

\begin{prop}\label{prop11.4} Let $\lambda_6$ denote $6$-dimensional Hausdorff measure. Then $\lambda_6(\tilde M) = 0$ a.s.
\end{prop}

\begin{proof} For $q$ large enough so that $\Omega_q$ occurs, on $\{\tau_{K,2} \wedge \tau_{K,3} = T_0 \}$, by the definition of $\zeta$ in \eqref{eq11.7} and Lemma \ref{lem11.2},
$$
 \sum_{A \in \cH_q} r_A^6 \leq \frac{1}{\log_2 q} \sum_{A \in \cH_q} \zeta(r_A) \leq \frac{K \lambda_2(R_0)}{\log_2 q} \to 0
$$
as $q \to \infty$. Since the family of balls $(B_A,\, A \in \cH_q)$ covers $\tilde M$ by Lemma \ref{lem11.3}, we conclude that $\lambda_6(\tilde M) = 0$ on $\Omega_q \cap \{\tau_{K,2} \wedge \tau_{K,3} = T_0 \}$. Since $\lim_{K \uparrow \infty} P\{\tau_{K,2} \wedge \tau_{K,3} = T_0 \} = 1$ and $\Omega_q$ occurs for large enough $q$ (by \eqref{eq11.2} and \eqref{eq11.6a}), the proposition is proved.
\end{proof}

\noindent{\em Proof of Theorem \ref{thm11.4}.} We first prove that $\lambda_6( M) = 0$, where
$$
   M= \{u(s,y): (s,y) \in R_0 \}.
$$
On the event $\{\tau_{K,1} = T_0 \}$, $u$ and $\tilde u$ coincide on $[0,T_0] \times \IR$, so $M = \tilde M$, where $\tilde M$ is defined in Lemma \ref{lem11.3}, and therefore, by Proposition \ref{prop11.4},
$$
   \lambda_6(M) = 0 \qquad\mbox{a.s.~on } \{\tau_{K,1} = T_0 \}.
$$
Since $\lim_{K \uparrow \infty} P\{\tau_{K,1} = T_0 \} = 1$, we conclude that $\lambda_6(M) = 0$ a.s.

Let $u(]0,\infty[\,\times \IR)$ denote the random set $\{u(s,y): (s,y) \in\, ]0,\infty[\,\times \IR  \}$. Since in the entire paper, the rectangle $R_0$ could have been replaced by any other compact rectangle in $]0,\infty[ \times \IR$, 
we deduce that $\lambda_6(u(]0,\infty[\,\times \IR)) = 0$. Therefore, for $d\geq 6$, $\lambda_d(u(]0,\infty[\,\times \IR)) = 0$, where $\lambda_d$ denotes Lebesgue-measure on $\IR^d$. By Fubini's theorem,
\begin{align*}
   0 = E\left[\int_{\IR^d} 1_{u(]0,\infty[\,\times \IR)}(z)\, \lambda_d(dz) \right] = \int_{\IR^d} P\{z \in u(]0,\infty[\,\times \IR) \} \, \lambda_d(dz),
\end{align*}
that is, for Lebesgue-almost all $z\in \IR^d$, $P\{z \in u(]0,\infty[\,\times \IR) \} = 0$. This proves Theorem \ref{thm11.4}.
\hfill $\Box$
\vskip 16pt

\noindent{\sc Acknowledgement.} The research reported in this paper was initiated at the Centre Interfacultaire Bernoulli, Ecole Polytechnique F\'ed\'erale de Lausanne, Switzerland, during the semester program ``Stochastic Analysis and Applications" in Spring 2012. We thank this intsitution for its hospitality and support.

\begin{thebibliography}{10}

\bibitem{BG}
R.~M. Blumenthal and R.~K. Getoor.
\newblock {\em Markov processes and potential theory}.
\newblock Pure and Applied Mathematics, Vol. 29. Academic Press, New
  York-London, 1968.

\bibitem{chen_dal}
Le~Chen and Robert~C. Dalang.
\newblock H\"{o}lder-continuity for the nonlinear stochastic heat equation with
  rough initial conditions.
\newblock {\em Stoch. Partial Differ. Equ. Anal. Comput.}, 2(3):316--352, 2014.

\bibitem{DZ}
Giuseppe Da~Prato and Jerzy Zabczyk.
\newblock {\em Stochastic equations in infinite dimensions}, volume 152 of {\em
  Encyclopedia of Mathematics and its Applications}.
\newblock Cambridge University Press, Cambridge, second edition, 2014.

\bibitem{DKLMX}
Robert Dalang, Slobodan Krstic, Cheuk~Yin Lee, Carl Mueller, and Yimin Xiao.
\newblock Multiple points of {G}aussian random fields via polarity of points.
\newblock In preparation.

\bibitem{DLMX}
Robert Dalang, Cheuk~Yin Lee, Carl Mueller, and Yimin Xiao.
\newblock Multiple points of {G}aussian random fields.
\newblock Preprint (2019).

\bibitem{dal_sanz}
Robert Dalang and Marta Sanz-Sol\'{e}.
\newblock Forthcoming book.

\bibitem{DKN1}
Robert~C. Dalang, Davar Khoshnevisan, and Eulalia Nualart.
\newblock Hitting probabilities for systems of non-linear stochastic heat
  equations with additive noise.
\newblock {\em ALEA Lat. Am. J. Probab. Math. Stat.}, 3:231--271, 2007.

\bibitem{DKN2}
Robert~C. Dalang, Davar Khoshnevisan, and Eulalia Nualart.
\newblock Hitting probabilities for systems for non-linear stochastic heat
  equations with multiplicative noise.
\newblock {\em Probab. Theory Related Fields}, 144(3-4):371--427, 2009.

\bibitem{DMX}
Robert~C. Dalang, Carl Mueller, and Yimin Xiao.
\newblock Polarity of points for {G}aussian random fields.
\newblock {\em Ann. Probab.}, 45(6B):4700--4751, 2017.

\bibitem{DN}
Robert~C. Dalang and Eulalia Nualart.
\newblock Potential theory for hyperbolic {SPDE}s.
\newblock {\em Ann. Probab.}, 32(3A):2099--2148, 2004.

\bibitem{dal_sanz1}
Robert~C. Dalang and Marta Sanz-Sol\'{e}.
\newblock Hitting probabilities for nonlinear systems of stochastic waves.
\newblock {\em Mem. Amer. Math. Soc.}, 237(1120):v+75, 2015.

\bibitem{doob}
Joseph~L. Doob.
\newblock {\em Classical potential theory and its probabilistic counterpart}.
\newblock Classics in Mathematics. Springer-Verlag, Berlin, 2001.
\newblock Reprint of the 1984 edition.

\bibitem{khosh2}
Davar Khoshnevisan.
\newblock {\em Multiparameter processes}.
\newblock Springer Monographs in Mathematics. Springer-Verlag, New York, 2002.
\newblock An introduction to random fields.

\bibitem{khosh1}
Davar Khoshnevisan.
\newblock A primer on stochastic partial differential equations.
\newblock In {\em A minicourse on stochastic partial differential equations},
  volume 1962 of {\em Lecture Notes in Math.}, pages 1--38. Springer, Berlin,
  2009.

\bibitem{ks}
Davar Khoshnevisan and Zhan Shi.
\newblock Brownian sheet and capacity.
\newblock {\em Ann. Probab.}, 27(3):1135--1159, 1999.

\bibitem{mueller_tribe}
C.~Mueller and R.~Tribe.
\newblock Hitting properties of a random string.
\newblock {\em Electron. J. Probab.}, 7:no. 10, 29, 2002.

\bibitem{mueller1}
Carl Mueller.
\newblock Some tools and results for parabolic stochastic partial differential
  equations.
\newblock In {\em A minicourse on stochastic partial differential equations},
  volume 1962 of {\em Lecture Notes in Math.}, pages 111--144. Springer,
  Berlin, 2009.

\bibitem{op73}
Steven Orey and William~E. Pruitt.
\newblock Sample functions of the {$N$}-parameter {W}iener process.
\newblock {\em Ann. Probability}, 1(1):138--163, 1973.

\bibitem{sanz-sarra}
Marta Sanz-Sol\'{e} and M\`onica Sarr\`a.
\newblock Path properties of a class of {G}aussian processes with applications
  to spde's.
\newblock In {\em Stochastic processes, physics and geometry: new interplays,
  {I} ({L}eipzig, 1999)}, volume~28 of {\em CMS Conf. Proc.}, pages 303--316.
  Amer. Math. Soc., Providence, RI, 2000.

\bibitem{T1}
Michel Talagrand.
\newblock Hausdorff measure of trajectories of multiparameter fractional
  {B}rownian motion.
\newblock {\em Ann. Probab.}, 23(2):767--775, 1995.

\bibitem{T2}
Michel Talagrand.
\newblock Multiple points of trajectories of multiparameter fractional
  {B}rownian motion.
\newblock {\em Probab. Theory Related Fields}, 112(4):545--563, 1998.

\bibitem{walsh}
John~B. Walsh.
\newblock An introduction to stochastic partial differential equations.
\newblock In {\em \'{E}cole d'\'{e}t\'{e} de probabilit\'{e}s de
  {S}aint-{F}lour, {XIV}---1984}, volume 1180 of {\em Lecture Notes in Math.},
  pages 265--439. Springer, Berlin, 1986.

\end{thebibliography}

Institut de Math\'ematiques

Ecole Polytechnique F\'ed\'erale de Lausanne

Station 8

CH-1015 Lausanne

Switzerland

Email: robert.dalang@epfl.ch
\\

Department of Mathematics

University of Rochester

Rochester, NY  14627

U.S.A.

Email: carl.e.mueller@rochester.edu
\\

Department of Statistics and Probability

Michigan State University

East Lansing, MI 48824,

USA

Email: xiao@stt.msu.edu

\end{document}